\documentclass[12pt]{amsart}
\usepackage{amssymb}

\theoremstyle{plain}
\newtheorem{prop}{Proposition}[section]
\newtheorem{teo}[prop]{Theorem}
\newtheorem{coro}[prop]{Corollary}
\newtheorem{lema}[prop]{Lemma}

\theoremstyle{definition}
\newtheorem{defi}[prop]{Definition}
\newtheorem{ejem}[prop]{Example}
\newtheorem{rem}[prop]{Remark}
\newtheorem{algorithm}[prop]{Algorithm}

\theoremstyle{remark}

\numberwithin{equation}{section}

\newcommand{\N}{\mathbb N}
\newcommand{\Z}{\mathbb Z}
\newcommand{\Q}{\mathbb Q}
\newcommand{\R}{\mathbb R}
\newcommand{\C}{\mathbb C}

\title[Isospectral manifolds]{Strongly isospectral manifolds with nonisomorphic cohomology rings}
\author{E. A. Lauret, R. J. Miatello and J. P. Rossetti}
\subjclass[2010]{58J53 (primary); 58C22, 20H15 (secondary).}
\keywords{isospectral, cohomology rings, primitive forms, flat manifolds}
\thanks{Supported by Conicet and Secyt-UNC}
\address{FaMAF--CIEM \\ Universidad Nacional de C\'ordoba.  5000 --- C\'ordoba, Argentina.}
\email{miatello@famaf.unc.edu.ar} \email{elauret@famaf.unc.edu.ar}\email{rossetti@famaf.unc.edu.ar}

\begin{document}
\maketitle

\begin{abstract}
For any $n\ge 7$, $k\ge 3$, we give pairs of compact flat $n$-manifolds  $M, M'$  with holonomy groups $\Z_2^k$, that are strongly isospectral, hence isospectral on $p$-forms for all values  of $p$, having nonisomorphic cohomology rings.
Moreover, if $n$ is even, $M$ is K\"ahler while  $M'$ is not. Furthermore, with the help of a computer program we show the existence of large Sunada isospectral families; for instance, for  $n= 24$ and $k=3$ there is a family of eight compact flat manifolds (four of them  K\"ahler) having very different cohomology rings. In particular, the cardinalities of the sets of primitive forms are different for all manifolds.
\end{abstract}

\section*{Introduction}\label{s.intro}
If $(M,g)$ is a compact Riemannian manifold and $0\le p \le n$, let $\textrm{spec}_p(M)$ denote the spectrum, with multiplicities, of the Hodge-Laplace operator acting on smooth $p$-forms on $(M,g)$. For each $p$, $\textrm{spec}_p(M)$  is  a sequence of non-negative real numbers tending to $\infty$.
If $\textrm{spec}_p(M)=\textrm{spec}_p(M')$, $(M,g)$ and $(M', g')$ are said to be {\it $p$-isospectral},  and just \emph{isospectral}, if $p=0$.

It has been known for quite some time that there exist manifolds that are isospectral on functions but not on 1-forms (see \cite{Go}, \cite{Ik}).
Also, C.~Gordon (see \cite{Go}) has given continuous families of pairs of nonisometric nilmanifolds that are isospectral on functions and not on 1-forms (here, the manifolds involved are homeomorphic to each other).

In the context of  compact flat manifolds, it turns out to be simpler to compute $p$-spectra and to determine some invariants, for instance, the Betti numbers.
In particular, there is a description of the cohomology ring as the ring of invariants of the holonomy action (H.~Hiller, \cite{Hi}).
In \cite{MR1} $p$-isospectrality is studied in this context and many new examples of $p$-isospectral nonhomeomorphic manifolds are given; in particular, pairs of manifolds $M, M'$ isospectral on functions such that $\beta_j(M) < \beta_j(M')$ for $1\le j\le n-1$. Hence such $M$ and $M'$ cannot be isospectral on $p$-forms for any $p \neq 0,n$ and are topologically quite different from each other, since they have different real cohomology rings.

The main goal of this paper  is to construct families of compact flat manifolds that are Sunada isospectral ---hence strongly isospectral--- but still their real cohomology rings are non-isomorphic to each other (see Theorem~\ref{thm:main}), despite the fact that they have the same Betti numbers. In particular, they are isospectral on $p$-forms for every $p$ but the ring structure of the cohomology rings may be very different.
The manifolds in question are obtained by using different free isometric actions of $\Z_2 ^k$ on $T^n = \Z^n \backslash \R^n$.
Furthermore, we shall see that for  $n$ even some of them are K\"ahler  while the others are not.

As a first step, we  show, in Theorem~\ref{thm:metodoflip}, a general procedure to construct pairs of almost-conjugate diagonal representations of $\Z_2^k$  with $k\ge 3$ (see Definition~\ref{def:almostconjugate}).
We describe diagonal representations by an $r$-tuple $q_1, q_2, \ldots, q_r$ with $\sum_1 ^r q_i = n$, where $q_j$ gives the multiplicity of the $j$-th character $\chi_j$.
We also give an algorithm that allows us to determine all families of almost-conjugate diagonal representations of $\Z_2^k$.
We implement it with the aid of computer programs for some small values of $k$ and $n$.
Tables~\ref{tabla:pairs3} and \ref{tabla:pairs4}  show all such pairs for $k=3$, $n\le 11$, and $k=4$, $n\le 10$ respectively.
In Table~\ref{tabla:flias3} we exhibit all families of cardinality at least three for $k=3$,  $n\le 15$.

A main tool in our study of the cohomology rings are the primitive invariant forms, i.e.\  those that cannot be obtained as wedge products of forms of lower degree.
In particular, we express the number of them in terms of the $r$-tuple of $q_j$'s (see Proposition~\ref{prop:prim}).
In Theorem~\ref{thm:generators} we show that this number coincides with the cardinality of a minimal generating set  of the cohomology ring, hence it is an invariant of the ring.
This is used in Section~\ref{sect:main} in the proof of the non-isomorphism of the cohomology rings of the strongly isospectral manifolds in our main result, Theorem~\ref{thm:main}.

In the last section, we exhibit many explicit examples of Sunada isospectral families.
We study in some detail a pair in dimension $n=8$, $M, M'$ such that $M$ is K\"ahler and $M'$ is not, giving the rings of invariants of both manifolds and comparing several aspects of the respective ring structures (this gives more examples answering a question in \cite[13.6, p. 657]{Be}).
In Example~\ref{ex:8dim24} we show a family of eight 24-dimensional manifolds, four of which are K\"ahler, showing that the numbers of primitive invariant forms of degree 4 are different for all eight manifolds, hence the cohomology rings cannot be isomorphic by Corollary~\ref{coro:sumprim}.

We include in Remark~\ref{question:open1} and Remark~\ref{question:open2} some open questions related to the results in this paper.

\section{Preliminaries}

\subsection*{Bieberbach groups}

A crystallographic group is a discrete cocompact subgroup  $\Gamma$ of the isometry group $I(\R^n)$ of $\R^n$.
If $\Gamma$ is torsion-free then $\Gamma$ is said to be a {\it Bieberbach group}.
Such $\Gamma$ acts properly discontinuously on $\R^n$, thus $M_\Gamma=\Gamma\backslash\R^n$ is a compact flat Riemannian manifold with fundamental group $\Gamma$. Furthermore, any such manifold arises in this way.
Since $I(\R^n)\cong \mathrm{O}(n) \rtimes \R^n$, any element $\gamma \in I(\R^n)$ decomposes uniquely as $\gamma = B L_b$, with $B \in \mathrm{O}(n)$ and $b\in \R^n$.
The translations in $\Gamma$ form a normal maximal abelian subgroup of finite index $L_\Lambda$ where $\Lambda$ is a lattice in $\R^n$ which is $B$-stable for each $BL_b \in \Gamma$.
The restriction to $\Gamma$ of the canonical projection from $I(\R^n)$ to $\mathrm{O}(n)$, given by $BL_b \mapsto B$, is a homomorphism with kernel $L_{\Lambda}$ and with image a finite subgroup of $\mathrm{O}(n)$, denoted by $F$ in this article, called the \emph{point group} of $\Gamma$.
Indeed, one has the exact sequence
\[
0\rightarrow L_\Lambda \rightarrow \Gamma \stackrel{r}{\rightarrow} F \rightarrow  1
\]
where $F$ is isomorphic to $L_\Lambda \backslash \Gamma$ and gives the linear holonomy group of the Riemannian manifold $M_\Gamma$.
The group $F$ acts on $\Lambda$ by an integral representation $\rho$ called the holonomy representation of $\Gamma$.

A Bieberbach group $\Gamma$ is said to be of \emph{diagonal type} if there exists an orthonormal $\Z$-basis $\{e_1,\dots,e_n\}$ of the lattice $\Lambda$ such that for any element $BL_b\in\Gamma$, $Be_i=\pm e_i$ for $1\le i\le n$.
These Bieberbach groups are those having the simplest holonomy action, among those with holonomy group $\Z_2^k$.
It is a useful fact that, for groups of diagonal type, after conjugation of $\Gamma$ by an isometry, it may be assumed that $\Lambda=\Z^n$ and, furthermore, that for any $\gamma=BL_b \in \Gamma$, $b$ lies in $\frac 1{2} \Z^n$.
Thus, any $\gamma \in \Gamma$  can be written uniquely as $\gamma=BL_{b_0}L_\lambda$, where the coordinates of $b_0$ are $0$ or $\frac 12$ and  $\lambda \in \Z^n$ (see~\cite[Lemma 1.4]{MR2}).

For $BL_b \in \Gamma$ define
\begin{align}
  n_B&= \textrm{dim}(\R^n)^B=|\{ 1\le i \le n: Be_i=e_i \}|\,, \label{eq:n_B}\\
  n_{B,\frac 12} & = |\{1 \le i \le n: Be_i=e_i \textrm{ and } b_0.e_i= \tfrac 12 \}|. \label{eq:n_Bhalf}
\end{align}
If $0\leq s\leq n$, let
\begin{align}\label{eq:c_s}
  c_{s}(F) &=\big|\big\{B \in F:n_B=s \big \}\big|.
\end{align}
If $0\le t \le s \le n$, the \emph{Sunada numbers} of $\Gamma$ are defined by
\begin{align}\label{eq:Sunada}
  c_{s,t}(\Gamma) &=\big|\big\{BL_b \in F:n_B=s \text{ and } n_{B,\frac 12}=t \big \}\big|.
\end{align}

It is a well-known fact that, by the torsion-free condition, $n_B\ge 1$ for any $BL_b \in \Gamma$. Clearly, $c_s(F)=\sum_t c_{s,t}(\Gamma)$.

\subsection*{$F$-invariants in exterior algebras}

As mentioned in the introduction, the cohomology ring over $\Q$ of a compact flat manifold $M_\Gamma$ with holonomy group $F$ can be computed by using the Hochschild-Serre spectral sequence, which gives
\begin{equation}\label{eq:HSseq}
H^*(M_\Gamma, \Q)\cong{\Lambda}^*_F(\Q^n)\,,
\end{equation}
the ring of $F$-invariants in the full exterior $\Q$-algebra ${\Lambda}^*(\Q^n)$ (see~\cite{Hi}).
In what follows we shall often abbreviate
\begin{eqnarray}\label{exterior}
\Lambda^*_F = \sum_{p=0}^n
\Lambda_F^p(\Q^n) \quad\text{and}\quad  \Lambda^p_F = \Lambda_F^p(\Q^n),
\end{eqnarray}
for $0\le p \le n$. In particular, $\dim(\Lambda^p_F) = \beta_p$ is the $p$-th Betti number of $M_{\Gamma}$.

We  mention some useful facts on the ring structure of $\Lambda^*_F$, for further use:
\begin{enumerate}
\item [(i)]  $\sum_{r=p}^n \Lambda_F^r$ is an ideal in $\Lambda^*_F$, for each $p$.
\item [(ii)] $\sum_{r=1}^n \Lambda_F^r$ is a maximal ideal of $\Lambda^*_F$ and any $\eta \in \Lambda^*_{F}\smallsetminus \sum_{r=1}^n  {\Lambda}_F^r $ is invertible.
\end{enumerate}
To verify the last claim, let $\eta = 1 +\delta \in \Lambda^*_F$ such that $\delta$ has degree zero component $\delta_0=0$. Then
\[
(1 +\delta )\wedge \big(\sum_{i=0}^n (-1)^i\; \underbrace{\delta \wedge\dots \wedge \delta}_i \big )=1
\]
and furthermore ${\sum_{i=0}^n (-1)^i \;\underbrace{\delta \wedge\dots \wedge \delta}_i \in \Lambda^*_F}$.

\subsection*{Primitive $F$-invariant forms}

From now on we assume that the subgroup $F$ of  $\mathrm{GL}_n(\Z)$ is of diagonal type,  i.e.\ $F$ is a group of diagonal matrices with $\pm 1$ in the diagonal, thus $F\cong \Z_2^k$ for some $1\le k \le n$.
If furthermore $F$ is the point group of a Bieberbach group, then $-\mathrm{Id}_n \notin F$ and $k\le n-1$. Here and subsequently, $\{e_1,\dots,e_n\}$ denotes the canonical basis of $\R^n$.

\begin{defi}\label{def:primitivos}
Let $F$ be a finite diagonal subgroup of $\mathrm{GL}_n(\Z)$. Given an ordered subset $I = \{i_1,\ldots ,i_p\} \subset \{1,\dots,n\}$ we set $e_I=e_{i_1}\wedge\ldots\wedge e_{i_p} \in \Lambda^p_F$.
The form $e_I$ is said to be \emph{primitive} if it cannot be obtained as a wedge product of $F$-invariant forms of degree lower than $p$.
We denote by ${\mathcal P}^p_F$ the set of all primitive forms of degree $p$, by $\Lambda^p_{F,prim}$ the span of ${\mathcal P}^p_F$ and by $P_{p,F}$ the cardinality of  ${\mathcal P}^p_F$.

Clearly, the set of all primitive forms is a set of generators of $\Lambda^*_F$ of cardinality $\sum_{p=0}^n P_{p,F}$.
We shall see that this is the minimal cardinality of any set of generators.
\end{defi}

We are interested in comparing  the $\Q$-algebras $\Lambda^*_F$ and $\Lambda^*_{F'}$ for two different Bieberbach groups $\Gamma$ and $\Gamma'$, having point groups $F, F'$, respectively.
The following result will be very useful to us.

\begin{teo}\label{thm:generators}
Given $F$ a finite diagonal subgroup of $\mathrm{GL}_n(\Z)$, let ${\mathcal G}$ be a set of generators of the algebra $\Lambda^*_F$.
Then $\# {\mathcal G} \ge  \sum_{p=0}^n  P_{p,F}$ and, if ${\mathcal G}$ is a minimal generating set, then $\# {\mathcal G}= \sum_{p=0}^n  P_{p,F}$.
\end{teo}
\begin{proof}
Denote by
\begin{equation}
{\mathcal G}_p = \left\{\eta  = {\textstyle \sum\limits_{r=p}^n} \, \eta_{r}\in {\mathcal G} : \textrm{ with } \eta_{r} \in \Lambda_F^r,\,  \eta_p \ne 0 \right\}\,,
\end{equation}
i.e.\ the elements in ${\mathcal G}$ having a nonzero component of minimal degree $p$.

We note that ${\mathcal G}_0$ is non-empty, otherwise we cannot obtain 1 as a sum of products of elements of  ${\mathcal G}$. Furthermore, given a set of elements in ${\mathcal G}_0$, by subtraction of a scalar multiple, we can eliminate the zero component of all but one of them.
Thus, we may  replace the initial generating set  ${\mathcal G}$ by another generating set of the same cardinality, such that ${\mathcal G}_0 = \{\eta^0\}$ has only one element with  $\eta^0_0=1$.

Now, the lowest components in $\Lambda^1_F$ of the elements in ${\mathcal G}_1$ must span all of $\Lambda^1_F$, otherwise there is no way we can generate $\Lambda_F^1$ with sums of products in ${\mathcal G}$.
Thus, we may select a subset ${\mathcal S}_1$ of  ${\mathcal G}_1$ with $\beta_1$ elements, such that the nonzero  components of minimal degree span $\Lambda^1_F$.
Furthermore, we can subtract a linear combination of elements in ${\mathcal S}_1$ from each element in ${\mathcal G}_1\smallsetminus {\mathcal S}_1$ to cancel the component in  $\Lambda_F^1$.

In this way, we may replace the original generating set ${\mathcal G}$ by another generating set of the same cardinality such that  $\textrm{card}({\mathcal G}_0) = 1$, $\textrm{card}({\mathcal G}_1) = P_{1,F}$.
Finally, by replacing the elements in ${\mathcal S}_1$  by linear combinations of them, we may further assume that their lowest degree terms run through the set of  $F$-invariant $e_j$'s, that is, through the set of all primitive forms  $e_j \in \Lambda_F^1$. Here note that $\beta_1 = P_{1,F}$.

In a similar way we assume inductively that we have replaced the original generating set ${\mathcal G}$ by another set of the same cardinality such that  the cardinality of ${\mathcal G}_r$ equals $P_{r,F}$ for each $r\le p$ and the lowest degree terms of the elements in ${\mathcal G}_r$, run through the set of primitive forms  $e_J \in \Lambda_F^r$ with $|J|=r$.

Now we consider the elements in ${\mathcal G}_{p+1}$.
Necessarily there must be at least $P_{p+1,F}$ of them, so that their $(p+1)$-components, together with sums of products of elements in ${\mathcal G}_r$ with $r \le p$ generate all of  $\Lambda^{p+1}_F$.
We may subtract from the elements in ${\mathcal G}_{p+1}$ linear combinations of wedge products of elements of smaller degree so that  their lowest order terms lie in the span of the space ${\mathcal P}^{p+1}_F$.
Actually, we  may fix a subset of cardinality $P_{p+1,F}$ such that their lowest order terms are a basis of  the space ${\mathcal P}^{p+1}_F$.
Finally by a linear algebra argument, we may change this set by one such that their lowest  order terms run exactly through the invariant forms  $e_J \in {\mathcal P}^{p+1}_F$.

In this way, in $n$ steps, we obtain a new set of generators of the $\Q$-algebra $\Lambda^*_F$ of cardinality  $\sum_{p=1}^n P_{p,F} \le  \textrm{card}({\mathcal G})$.
Clearly if ${\mathcal G}$ is minimal, then  $\sum_{p=1}^n P_{p,F} =  \textrm{card}({\mathcal G})$.
This completes the proof of the theorem.
\end{proof}

\begin{coro}\label{coro:sumprim}
Let $F, F'$  finite diagonal subgroups of $\mathrm{GL}_n(\mathbb Z)$.
If, as $\Q$-algebras, $\Lambda^*_F \cong \Lambda^*_{F'}$ then
\begin{equation}
\sum_{p=1}^n P_{p,F} = \sum_{p=1}^n P_{p,F'}\,.
\end{equation}
If as graded $\Q$-algebras $\Lambda^*_F \cong \Lambda^*_{F'}$ then $P_{p,F} = P_{p,F'}$ for every $0\le p \le n$.
\end{coro}
\begin{proof}
By the previous theorem, $\sum_{p=1}^n P_{p,F}$ is the cardinality of a minimal generating set in $\Lambda^*_F$, hence it must be invariant under isomorphisms. The second assertion is also clear.
\end{proof}

\section{Construction of almost-conjugate representations}\label{s.construction}

This section is devoted to the construction of pairs of almost-conjugate representations which give the point groups of certain Bieberbach groups of diagonal type.
The corresponding pairs of manifolds, to be constructed in Sections~\ref{sect:main} and \ref{sect:examples}, will be Sunada isospectral (\cite{Go2} or \cite{MR3}) and they will have different cohomology rings.

\begin{defi}\label{def:diag-rep}
A monomorphism $\rho:\Z_2^k\to \textrm{GL}_n(\Z)$ such that $\mathrm{Im} (\rho)$ is a subgroup of diagonal matrices will be called an \emph{integral diagonal representation} of $\Z_2^k$  or, for brevity, a \emph{diagonal representation} of $\Z_2^k$.
\end{defi}

A \emph{character} of $\Z_2^k$ is a homomorphism $\chi: \Z_2^k \to \{\pm 1\}$. The set of all such characters is denoted by $\widehat{\Z_2^k} \cong \Z_2^k$. Sometimes, it will be convenient to identify characters of $ \Z_2^k$ with subsets of $\{1,\dots,k\}$.
If $f_1,\ldots,f_k$ denotes the canonical basis of $\Z_2^k$, for $I\subset \{1,\dots,k\}$ we set  $\chi_{_I}:\Z_2^k\to\{\pm1\}$, the character given on  basis elements by
\[
\chi_{_I}(f_i)=
\begin{cases}
  -1&\textrm{if $i\in I$,}\\
 \phantom {-}1&\textrm{if $i\notin I$,}
\end{cases}
\qquad \textrm{for }1\leq i\leq k.
\]
Thus $\chi_{_{I_1}}\chi_{_{I_2}}=\chi_{_{I_1\triangle I_2}}$ for $I_1,I_2\subset\{1,\dots,k\}$, where $I_1\triangle I_2=(I_1\cup I_2)\smallsetminus (I_1\cap I_2)$ denotes the symmetric difference of sets.

From now on it will be convenient to fix  a total order $\prec$ on ${\widehat\Z_2^k}$ (or equivalently on the subsets of $\{1,\dots,k\}$) with the only requirement that  $\chi_\emptyset=1$ is the first element.

Any $n$-dimensional diagonal representations of $\Z_2^k$ can be decomposed as a sum $\rho=\sum_I q_{_I} \chi_{_I}$, with $q_{_I} \in \N_0:=\N\cup\{0\}$ and $n = \sum_I {q}_{_I}$,  the sum  running over all subsets of $\{1,\dots,k\}$.
Conversely, if $r=2^k $, for each choice of numbers $q_{_{I}} \in \N_0$, we define the diagonal representation $\rho=\sum_I q_{_I} \chi_{_I}$  such that
\begin{equation}\label{eq:rho}
\rho(f)= \mathrm{diag}\Big(\underbrace{\chi_{_{J_1}}(f), \dots ,\chi_{_{J_1}}(f)}_{q_{_{J_1}}}, \ldots, \underbrace{\chi_{_{J_r}}(f),\ldots, \chi_{_{J_{r}}}(f)}_{q_{_{J_r}}}\Big)\,,
\end{equation}
for $f \in \Z_2^k$, where the characters $\chi_{_{J_i}}$ are ordered by  $\prec$.

\begin{defi} \label{def:F}
Let $\rho$ be a diagonal representation of $\Z_2^k= \langle f_j: 1\le j \le k\rangle$.
We will denote by $F$ the image of $\rho$, $F = \mathrm{Im}(\rho)\cong \Z_2^k$, which is generated by the diagonal matrices $B_i := \rho(f_i)$, $1\le i \le k$.
Given a subset $I=\{i_1, \dots, i_h\} \subset \{1,\dots,k\}$, we set $B_I = B_{i_1}\dots B_{i_h}$.
For simplicity, we will often write
\begin{equation}
B_{i_1\dots i_h},\, q_{i_1\dots i_h},\, \chi_{i_1\dots i_h},\, q_0,\, \chi_0
\quad\text{in place of}\quad
B_I,\, q_{_I},\, \chi_{_I},\, q_\emptyset,\,\chi_\emptyset
\end{equation}
respectively.
\end{defi}

Let $\rho$ and $\rho'$ be diagonal representations of $\Z_2^k$.
It is easy to check that these representations are equivalent if and only if the groups $F$ and $F'$ are conjugate in $\mathrm{O}(n)$.
For example, $\rho=2\chi_1+\chi_2+\chi_{12}$ and $\rho'=\chi_1+2\chi_2+\chi_{12}$ are two equivalent diagonal representation of $\Z_2^3$.

We will need some more notation. For $p\geq1$, let
\begin{multline}\label{eq:Ap}
{\mathcal A_p} = \big\{ \{I_1,\dots, I_p\} \,:\, \chi_{_{I_1}}\ldots\chi_{_{I_p}} =1 \;  \textrm{ and no proper subproduct} \\  \textrm{ of the }\chi_{_{I_p}} \textrm{  equals } 1 \big\}\,.
\end{multline}
For example, for $k=3$, since there are seven nontrivial characters, and one can check that $\#\mathcal A_0=\#\mathcal A_1=1$, $\#\mathcal A_2=\#\mathcal A_3=\#\mathcal A_4=7$ and $\mathcal A_p =\emptyset$ for any other~$p$.

The next proposition  gives some formulas to be used in the next section to compute the primitive elements of certain diagonal Bieberbach groups.

\begin{prop}\label{prop:prim}
Let $\rho=\sum_{I} q_{_I}\chi_{_I}$ be a diagonal representation of $\Z_2^k$.

\textup{(i)} The number $P_{p,F}$  of primitive $F$-invariant forms  of degree  $p$ (see Definition~\ref{def:primitivos}) is given by the following expression:
\begin{equation}\label{eq:P_p}
\begin{array}{l}
\displaystyle
P_{0,F}= 1,\quad\quad\;
P_{1,F}= q_{\emptyset}, \quad\quad\;
P_{2,F}= \sum_{I\neq\emptyset} \binom{q_{_I}}{2},\\
\displaystyle
P_{p,F}= \sum_{\{I_1, \dots, I_p\} \in \mathcal A_p} q_{_{I_1}}\dots q_{_{I_p}}\qquad \textrm{for}\quad 3 \le p \le k,
\end{array}
\end{equation}
where ${\mathcal A}_p$ is as in \eqref{eq:Ap}.
Moreover, $P_{p,F} = 0$ for any $p>k+1$.

\textup{(ii)} If $k=3$ one has
\begin{multline}\label{eq:prim4}
\;\; P_{4,F}
= q_{1}\, q_{2}\,  q_{3}\,  q_{123}
+ q_{1}\, q_{2}\,  q_{13}\, q_{23}
+ q_{1}\, q_{3}\,  q_{12}\, q_{23} \\
+ q_{1}\, q_{12}\, q_{13}\, q_{123}
+ q_{2}\, q_{3}\,  q_{12}\, q_{13}
+ q_{2}\, q_{12}\, q_{23}\, q_{123}
+ q_{3}\, q_{13}\, q_{23}\, q_{123}.
\end{multline}

\textup{(iii)} Since $\beta_p =\dim \Lambda_F^p$ for $0\le p \le n$, then
\[
\begin{array}{ll}
\displaystyle
\beta_0 = P_{0,F} = 1, \hspace{2cm}  &
\displaystyle
\beta_1 = P_{1,F}, \\[3mm]
\displaystyle
\beta_2 = \tbinom {q_{\emptyset}}{2} + P_{2,F}, \hspace{2cm} &
\displaystyle
\beta_3 = \tbinom{q_\emptyset}3 + q_\emptyset P_{2,F}  +  P_{3,F}, \\[3mm]
\multicolumn{2}{l}{ \displaystyle
\beta_4 = \sum_{\emptyset\neq I} \tbinom{q_{_I}}{4} + \sum_{\emptyset \neq I_1\prec I_2} \tbinom{q_{_{I_1}}}2 \tbinom{q_{_{I_2}}}2
           + q_\emptyset P_{3,F} + \tbinom{q_\emptyset}2 P_{2,F} + \tbinom{q_\emptyset}4 + P_{4,F}.}
\end{array}
\]
In particular, if $q_{\emptyset}=\beta_1=0$, then
\begin{equation}
\beta_2 = P_{2,F}, \quad\quad
\beta_3 = P_{3,F}, \quad\quad
\beta_4 = \sum_{\emptyset \neq I_1\prec I_2} \tbinom{q_{_{I_1}}}2 \tbinom{q_{_{I_2}}}2 +  P_{4,F} .
\end{equation}
\end{prop}
\begin{proof}
Clearly $P_{0,F}=1$.
By \eqref{eq:rho},  the group $F$ acts by a character $\psi_j$ on each $e_j\in \R^n$. Hence, for any $p$, the indecomposable invariant forms of degree $p$ are of the form $e_{i_1}\wedge e_{i_2}\ldots \wedge e_{i_p}$ where the corresponding characters satisfy $ \psi_{{i_1}}\, \psi_{{i_2}}\dots  \psi_{{i_p}}=1$ and  they are primitive if and only if  none of the proper subproducts of the $ \psi_{i_j}$ equals one.
Now, this is clearly equivalent to $\{I_1,\dots,I_p\}\in\mathcal A_h$, and consequently \eqref{eq:P_p} follows.
Furthermore, in this situation, it is necessary  that $\psi_{i_1}, \psi_{i_2} \ldots, \psi_{i_{p-1}}$ be linearly independent, hence $p-1 \le k$, as claimed in (i).

We now prove (ii). By the definition we have
\begin{multline*}
\mathcal A_4 =
\Big\{
\{1,2,3,123 \}, \{1,2,13,23\}, \{1,3,12,23\}, \{1,12,13,123\}, \\ \{2,3,12,13\}, \{2,12,23,123\}, \{3,13,23,123\}
\Big\},
\end{multline*}
where we have written $i_1\dots i_p$ in place of $\{i_1,\dots,i_p\}$. The asserted expression of $P_{4,F}$  follows immediately from \eqref{eq:P_p}.

The  expressions in (iii) can be easily obtained from (i).
\end{proof}

The following notion will be useful in the construction of isospectral flat manifolds.
\begin{defi}\label{def:almostconjugate}
We say that two diagonal representations $\rho$ and $\rho'$ are \emph{almost-conjugate} if the subgroups $F=\textrm{Im} (\rho)$ and $F'=\textrm{Im} (\rho')$ are almost-conjugate, that is, if there is a bijection $\phi: F \to F'$ that preserves the conjugacy class in $\mathrm{O}(n)$.
\end{defi}

Note that since the only eigenvalues of the elements of $F$ and $F'$ are $\pm 1$, the condition in the definition is equivalent to requiring that, for each  $0\leq s\leq n$,
\begin{equation}\label{eq:patron}
c_s(F)=c_s(F')\, ,
\end{equation}
in the notation of \eqref{eq:c_s}.

When the bijection $\phi$ is an isomorphism, then the representations $\rho$ and $\rho'$ are actually equivalent, but in general this is not the case (this can be easily checked in Example~\ref{ex:dim8} by using \eqref{eq:patron}).

Given $\rho=\sum_I q_{_I}\chi_{_I}$, our next goal is to perform a small perturbation of $\rho$ by constructing a diagonal representation $\rho'$ of $\Z_2^k$ having the same set of $n_B$'s with their multiplicities, more precisely, $\rho'$ will satisfy $n_{B_I'}=n_{B_I}$ for every $I\subset\{1,\dots,k\}$ with $I\neq\{1\},\{2\}$, $n_{B_1'}=n_{B_2}$ and $n_{B_2'}=n_{B_1}$.
Thus $\rho$ and $\rho'$ will be almost-conjugate. We assume first that $k=3$ and $\rho$ is a fixed diagonal representation of $\Z_2^3$.
The above equations induce a  linear system of eight equations in the eight variables $q_0', q_1',\dots,q_{123}'$ that turns out to be non-singular.
Setting $u=q_1'-q_1$, we can write the solution of the system as
\begin{align*}
q_0'&=q_0, & q_1'&=q_1+u, & q_{13}'&=q_{13}+u, & q_{12}'&=q_{12},\\
q_3'&=q_3, & q_2'&=q_2-u, & q_{23}'&=q_{23}-u, & q_{123}'&=q_{123},
\end{align*}
under the condition
\[
q_1'+q_{13}'+2u=q_2'+q_{23}'.
\]
Note that $\rho'$ will be a solution such that $q_I'\in\N_0$ for all $I$  if and only if  $u=(q_2'+q_{23}'-q_1'-q_{13}')/2\in\Z$ and $q_1+u, q_2-u, q_{13}+u, q_{23}-u\in\N_0$.

As we shall see in Theorem~\ref{thm:metodoflip}, this method generalizes to any $k\geq3$ and gives a procedure to construct pairs of almost-conjugate representations.

\begin{defi}\label{def:flip}
Given $\rho=\sum_{I} q_{_I}\chi_{_I}$ a diagonal representation of $\Z_2^k$ such that the number
\begin{equation}\label{eq:u}
  u=\frac{1}{2^{k-2}}\left(\sum_{2\in I,\, 1\notin I} q_{_I} - \sum_{1\in I,\, 2\notin I} q_{_I}\right)
\end{equation}
is an integer and furthermore $q_{_I}-u\geq 0\,$ if $\,2\in I,\, 1\notin I$ and  $q_{_I}+u\geq 0\,$ if $1\in I,\, 2\notin I$, we define the \emph{flip} of $\rho$ as
\begin{equation}\label{eq:sigma}
  \rho'=\sum_{I} (q_{_I}+u\delta_{_I})\chi_{_I}, \quad \textrm{ where } \quad \delta_I:=
  \begin{cases}
  1&\textrm{ if }\;  1\in I,\, 2\notin I,\\
 -1&\textrm{ if }\;  2\in I,\, 1\notin I,\\
  0&\textrm{ otherwise.}
 \end{cases}
\end{equation}
\end{defi}

It is easy to check that $\rho'$ is again a diagonal representation of $\Z_2^k$ in the sense of Definition~\ref{def:diag-rep}.

\begin{teo}\label{thm:metodoflip}
Let $\rho=\sum_{I} q_{_I}\chi_{_I}$ be a diagonal representation of $\Z_2^k$. If $\rho'$ is the flip of $\rho$ then
\begin{equation}\label{eq:n_Bflip}
n_{B_{1}}=n_{B_{2}'},\qquad
n_{B_{2}}=n_{B_{1}'},\qquad
n_{B_I}  =n_{B_I'} \quad\text{for }\; I\neq\{1\},\{2\}.
\end{equation}
In particular, $\rho$ and $\rho'$ are almost-conjugate representations.
\end{teo}
\begin{proof}
It suffices to verify that \eqref{eq:n_Bflip} holds for the flip $\rho'$ of $\rho$.
We will use the following facts. First, for each  $I\subset\{1,\dots,k\}$ we have that
\begin{align}\label{eq:n_Bfla}
  n_{B_I }&=\sum_{J:\chi_{_J}(f_I)=1} q_{_J}, &
  n_{B_I'}&=\sum_{J:\chi_{_J}(f_I)=1} (q_{_J} +u\delta_{_J}).
\end{align}
Secondly, if $I=\{i_1,\dots,i_s\}$ then
\begin{equation}\label{eq:chi_J(f_I)}
\chi_{_J}(f_I)=\chi_{_J}(f_{i_1})\dots \chi_{_J}(f_{i_s})=(-1)^{\#(J\cap I)},
\end{equation}
since $\chi_{_J}(f_{j})=-1$ if and only if $j\in J$.

Now, using \eqref{eq:sigma}, \eqref{eq:n_Bfla} and \eqref{eq:chi_J(f_I)} we have that
\begin{align*}
n_{B_2'}&= \sum_{2\notin J} (q_{_J}+u\delta_{_J})
         = \sum_{1\notin J,\, 2\notin J} (q_{_J}+ {u\delta_{_J}}) + \sum_{1\in J, \, 2\notin J} (q_{_J}+{u\delta_{_J}})\\
        &= \sum_{1\notin J,\, 2\notin J} q_{_J} + \sum_{1\in J, \, 2\notin J} q_{_J}+ 2^{k-2}u,
\end{align*}
since there are exactly $2^{k-2}$ subsets of $\{1,\dots,k\}$ containing $1$ and not $2$.
Using \eqref{eq:u} we conclude that
\begin{align*}
n_{B_2'} = \sum_{1\notin J,\, 2\notin J} q_{_J} + \sum_{1\notin J,\, 2\in J} q_{_J} = \sum_{1\notin J}q_{_J} = n_{B_1}.
\end{align*}
By arguing in the same way we can check that $n_{B_{1}'}=n_{B_{2}}$.

Now, for $I\subset\{1,\dots,k\}$, by \eqref{eq:n_Bfla} it follows that
\begin{align*}
n_{B_{I}'} - n_{B_{I}}&=\sum_{\chi_{_J}(f_I)=1} u\delta_{_J} =u\left(\sum_{\chi_{_J}(f_I)=1\atop 1\in J,\, 2\notin J} 1 -\sum_{\chi_{_J}(f_I)=1\atop 1\notin J,\, 2\in J} 1\right).
\end{align*}
It is now easily seen that if $I\neq\{1\},\{2\}$, the sums in the right-hand side are both equal to $2^{k-2}$.
This completes the proof of the theorem.
\end{proof}

\begin{rem}
(i) We note that it is possible to use any pair $I_1,\;I_2$ of nonempty subsets of $\{1,\dots,k\}$ in place of $I_1=\{1\},\;I_2=\{2\}$, to produce a flip.
Seldom it may be possible to apply two or more different flips to a representation, producing families of almost-conjugate representations.

(ii) If $u=0$ then $\rho$ and its flip representation $\rho'$ coincide.
When $u\neq0$, in some rare cases the representations $\rho$ and $\rho'$ may turn out to be equivalent.
To show an example, if we choose $\rho=\chi_{1}+2\chi_{2}+\chi_{3}+\chi_{12}+\chi_{23}$, we obtain $\rho'=2\chi_{1}+\chi_{2}+\chi_{3}+\chi_{12}+\chi_{13}$ which is equivalent to $\rho$.
\end{rem}

\newcommand{\espaciocol}{\hspace{5mm}}
\newcommand{\FF}[1]{\!\!\mathcal F^{3,{#1}}_}
\begin{table}[hbt]
\caption{All families $\mathcal F^{k,{n}}_j$ of almost-conjugate diagonal representations of $\Z_2^k$ of dimension $n\leq11$.}
\label{tabla:pairs3}
\vspace{3mm}
$
\begin{array}{c@{\espaciocol}c@{}l@{\espaciocol}c@{}l@{\espaciocol}c@{}l@{\,}}
n&&\hfill\textrm{Reps.}\hfill P_{4,F}\,&&\hfill\textrm{Reps.}\hfill P_{4,F}\,&&\hfill\textrm{Reps.}\hfill P_{4,F}\,\\ \hline\hline
%
%
7&
\FF71&
 \begin{array}{@{[}c@{,}c@{,}c@{,}c@{,}c@{,}c@{,}c@{]\;\;}l}
 3 & 1 & 1 & 1 & 0 & 1 & 0 & 3 \\ 2 & 2 & 2 & 1 & 0 & 0 & 0 & 0
 \end{array}&&&&\\ \hline\hline
%
%
8&
\FF81 &
 \begin{array}{@{[}c@{,}c@{,}c@{,}c@{,}c@{,}c@{,}c@{]\;\;}l}
 3 & 2 & 1 & 1 & 0 & 1 & 0 & 3 \\ 2 & 2 & 2 & 2 & 0 & 0 & 0 & 0
 \end{array}   &
\FF82 &
 \begin{array}{@{[}c@{,}c@{,}c@{,}c@{,}c@{,}c@{,}c@{]\;\;}l}
 3 & 1 & 1 & 1 & 1 & 1 & 0 & 7 \\ 2 & 2 & 2 & 1 & 1 & 0 & 0 & 4
 \end{array}&&\\ \hline\hline
%
%
&
\FF91 &
   \begin{array}{@{[}c@{,}c@{,}c@{,}c@{,}c@{,}c@{,}c@{]\;\;}l}
   4 & 2 & 1 & 1 & 0 & 0 & 1 & 8 \\ 3 & 3 & 2 & 1 & 0 & 0 & 0 & 0
   \end{array}   &
\FF93 &
   \begin{array}{@{[}c@{,}c@{,}c@{,}c@{,}c@{,}c@{,}c@{]\;\;}l}
   3 & 3 & 1 & 1 & 1 & 0 & 0 & 3 \\ 3 & 2 & 2 & 2 & 0 & 0 & 0 & 0
   \end{array}&
\FF95 &
   \begin{array}{@{[}c@{,}c@{,}c@{,}c@{,}c@{,}c@{,}c@{]\;\;}l}
   3 & 1 & 1 & 1 & 1 & 1 & 1 & 15 \\ 2 & 2 & 2 & 1 & 1 & 1 & 0 & 12
   \end{array}\\ \cline{2-7}
\raisebox{2.5ex}[0pt]{9}&
\FF92 &
   \begin{array}{@{[}c@{,}c@{,}c@{,}c@{,}c@{,}c@{,}c@{]\;\;}l}
   4 & 2 & 1 & 0 & 1 & 1 & 0 & 8 \\ 3 & 3 & 2 & 0 & 1 & 0 & 0 & 0
   \end{array}&
\FF94 &
   \begin{array}{@{[}c@{,}c@{,}c@{,}c@{,}c@{,}c@{,}c@{]\;\;}l}
   3 & 2 & 1 & 1 & 1 & 1 & 0 & 11 \\ 2 & 2 & 2 & 2 & 1 & 0 & 0 & 8
   \end{array}&&\\ \hline\hline
%
%
&
\FF{10}1 &
   \begin{array}{@{[}c@{,}c@{,}c@{,}c@{,}c@{,}c@{,}c@{]\;\;}l}
    4 & 3 & 1 & 1 & 1 & 0 & 0 & 3 \\ 4 & 2 & 2 & 2 & 0 & 0 & 0 & 0
   \end{array} &
\FF{10}4 &
   \begin{array}{@{[}c@{,}c@{,}c@{,}c@{,}c@{,}c@{,}c@{]\;\;}l}
    4 & 2 & 1 & 1 & 1 & 1 & 0 & 14 \\ 3 & 3 & 2 & 1 & 1 & 0 & 0 & 6
   \end{array} &
\FF{10}7 &
   \begin{array}{@{[}c@{,}c@{,}c@{,}c@{,}c@{,}c@{,}c@{]\;\;}l}
    3 & 2 & 2 & 1 & 1 & 0 & 1 & 19 \\ 2 & 2 & 2 & 2 & 2 & 0 & 0 & 16
   \end{array} \\ \cline{2-7}
10&
\FF{10}2 &
    \begin{array}{@{[}c@{,}c@{,}c@{,}c@{,}c@{,}c@{,}c@{]\;\;}l}
    4 & 2 & 2 & 1 & 0 & 1 & 0 & 8 \\ 3 & 3 & 2 & 0 & 2 & 0 & 0 & 0
   \end{array} &
\FF{10}5 &
    \begin{array}{@{[}c@{,}c@{,}c@{,}c@{,}c@{,}c@{,}c@{]\;\;}l}
    4 & 2 & 1 & 0 & 1 & 1 & 1 & 17 \\ 3 & 3 & 2 & 0 & 1 & 1 & 0 & 9
   \end{array} &
\FF{10}8 &
    \begin{array}{@{[}c@{,}c@{,}c@{,}c@{,}c@{,}c@{,}c@{]\;\;}l}
     3 & 2 & 1 & 1 & 1 & 1 & 1 & 23 \\ 2 & 2 & 2 & 2 & 1 & 1 & 0 & 20
    \end{array} \\ \cline{2-7}
&
\FF{10}3 &
    \begin{array}{@{[}c@{,}c@{,}c@{,}c@{,}c@{,}c@{,}c@{]\;\;}l}
    4 & 2 & 1 & 2 & 0 & 1 & 0 & 8 \\ 3 & 3 & 2 & 2 & 0 & 0 & 0 & 0
   \end{array} &
\FF{10}6 &
    \begin{array}{@{[}c@{,}c@{,}c@{,}c@{,}c@{,}c@{,}c@{]\;\;}l}
    3 & 3 & 1 & 1 & 1 & 1 & 0 & 15 \\ 3 & 2 & 2 & 2 & 0 & 1 & 0 & 12
   \end{array} &&\\ \hline\hline
%
%
&
\FF{11}1 &
  \begin{array}{@{[}c@{,}c@{,}c@{,}c@{,}c@{,}c@{,}c@{]\;\;}l}
  5 & 3 & 1 & 1 & 1 & 0 & 0 & 3 \\ 5 & 2 & 2 & 2 & 0 & 0 & 0 & 0
  \end{array} &
\FF{11}7 &
  \begin{array}{@{[}c@{,}c@{,}c@{,}c@{,}c@{,}c@{,}c@{]\;\;}l}
   4 & 3 & 1 & 2 & 0 & 1 & 0 & 8 \\ 3 & 3 & 2 & 3 & 0 & 0 & 0 & 0
 \end{array} &
\FF{11}{12} &
  \begin{array}{@{[}c@{,}c@{,}c@{,}c@{,}c@{,}c@{,}c@{]\;\;}l}
   4 & 2 & 2 & 1 & 0 & 1 & 1 & 26 \\3 & 3 & 2 & 0 & 2 & 1 & 0 & 18
   \end{array} \\ \cline{2-7}
&
\FF{11}2 &
  \begin{array}{@{[}c@{,}c@{,}c@{,}c@{,}c@{,}c@{,}c@{]\;\;}l}
  5 & 3 & 1 & 1 & 0 & 0 & 1 & 15\\ 4 & 4 & 2 & 1 & 0 & 0 & 0 & 0
  \end{array} &
\FF{11}8 &
  \begin{array}{@{[}c@{,}c@{,}c@{,}c@{,}c@{,}c@{,}c@{]\;\;}l}
   4 & 3 & 1 & 1 & 1 & 1 & 0 & 19 \\ 4 & 2 & 2 & 2 & 0 & 1 & 0 & 16
   \end{array} &
\FF{11}{13} &
  \begin{array}{@{[}c@{,}c@{,}c@{,}c@{,}c@{,}c@{,}c@{]\;\;}l}
   4 & 2 & 1 & 1 & 1 & 1 & 1 & 29 \\3 & 3 & 2 & 1 & 1 & 1 & 0 & 21
   \end{array} \\ \cline{2-7}
&
\FF{11}3 &
  \begin{array}{@{[}c@{,}c@{,}c@{,}c@{,}c@{,}c@{,}c@{]\;\;}l}
   5 & 3 & 1 & 0 & 1 & 1 & 0 & 15 \\ 4 & 4 & 2 & 0 & 1 & 0 & 0 & 0
   \end{array} &
\FF{11}9 &
  \begin{array}{@{[}c@{,}c@{,}c@{,}c@{,}c@{,}c@{,}c@{]\;\;}l}
   4 & 2 & 2 & 1 & 0 & 0 & 2 & 32 \\3 & 3 & 3 & 1 & 0 & 0 & 1 & 27
   \end{array} &
\FF{11}{14} &
  \begin{array}{@{[}c@{,}c@{,}c@{,}c@{,}c@{,}c@{,}c@{]\;\;}l}
   3 & 3 & 2 & 1 & 1 & 0 & 1 & 27 \\ 3 & 2 & 2 & 2 & 0 & 2 & 0 & 24
   \end{array} \\ \cline{2-7}
\raisebox{2.5ex}[0pt]{11}&
\FF{11}4 &
  \begin{array}{@{[}c@{,}c@{,}c@{,}c@{,}c@{,}c@{,}c@{]\;\;}l}
   5 & 2 & 2 & 1 & 0 & 0 & 1 & 20 \\ 4 & 3 & 3 & 1 & 0 & 0 & 0 & 0
   \end{array} &
\FF{11}{10} &
  \begin{array}{@{[}c@{,}c@{,}c@{,}c@{,}c@{,}c@{,}c@{]\;\;}l}
   4 & 2 & 1 & 2 & 1 & 1 & 0 & 20 \\ 3 & 3 & 2 & 2 & 1 & 0 & 0  & 12
   \end{array} &
\FF{11}{15} &
  \begin{array}{@{[}c@{,}c@{,}c@{,}c@{,}c@{,}c@{,}c@{]\;\;}l}
   3 & 3 & 1 & 1 & 1 & 1 & 1 & 31 \\ 3 & 2 & 2 & 2 & 0 & 1 & 1 & 28
   \end{array} \\ \cline{2-7}
&
\FF{11}5 &
  \begin{array}{@{[}c@{,}c@{,}c@{,}c@{,}c@{,}c@{,}c@{]\;\;}l}
   5 & 2 & 2 & 0 & 0 & 1 & 1 & 20 \\ 4 & 3 & 3 & 0 & 0 & 1 & 0 & 0
   \end{array} &
\FF{11}{11} &
  \begin{array}{@{[}c@{,}c@{,}c@{,}c@{,}c@{,}c@{,}c@{]\;\;}l}
   4 & 2 & 2 & 1 & 1 & 1 & 0 & 20 \\ 3 & 3 & 2 & 1 & 2 & 0 & 0  & 12
   \end{array} &
\FF{11}{16} &
  \begin{array}{@{[}c@{,}c@{,}c@{,}c@{,}c@{,}c@{,}c@{]\;\;}l}
   3 & 2 & 2 & 1 & 1 & 1 & 1 & 35 \\ 2 & 2 & 2 & 2 & 2 & 1 & 0 & 32
   \end{array}\\ \cline{2-7}
&
\FF{11}6 &
  \begin{array}{@{[}c@{,}c@{,}c@{,}c@{,}c@{,}c@{,}c@{]\;\;}l}
   4 & 3 & 2 & 1 & 0 & 1 & 0 & 8 \\ 3 & 3 & 3 & 2 & 0 & 0 & 0 & 0
   \end{array} &&&&\\ \hline
\end{array}
$
\end{table}

\newcommand{\espacio}{\hspace{1ex}}
\renewcommand{\FF}[1]{\!\!\widetilde{\mathcal F}^{3,{#1}}_}
\begin{table}[hbt]
\caption{All families $\widetilde{\mathcal F}^{k,n}_j$ with more than two elements of almost-conjugate diagonal representations of $\Z_2^3$ of dimension $n\leq15$.}
\label{tabla:flias3}
\vspace{3mm}
$
\begin{array}{c@{\espaciocol}c@{}l@{\espaciocol}c@{}l@{\espaciocol}c@{}l@{\,}}
n&&\hfill\textrm{Reps.}\hfill P_4\;&&\hfill\textrm{Reps.}\hfill P_4\;&&\hfill\textrm{Reps.}\hfill P_4\; \\ \hline\hline
%
%
12&
\FF{12}{1}&
\begin{array}{@{[}c@{,}c@{,}c@{,}c@{,}c@{,}c@{,}c@{]\;\;}r}
5 & 3 & 1 & 1 & 1 & 1 & 0 & 23 \\
5 & 2 & 2 & 2 & 0 & 1 & 0 & 20 \\
4 & 4 & 2 & 1 & 1 & 0 & 0 & 8 \\
4 & 3 & 3 & 2 & 0 & 0 & 0 & 0
\end{array}&
\FF{12}{2}&
\begin{array}{@{[}c@{,}c@{,}c@{,}c@{,}c@{,}c@{,}c@{]\;\;}r}
4 & 3 & 2 & 1 & 0 & 1 & 1 & 35 \\
4 & 2 & 2 & 2 & 0 & 2 & 0 & 32 \\
3 & 3 & 3 & 2 & 0 & 0 & 1 & 27
\end{array}&
\FF{12}{3}&
\begin{array}{@{[}c@{,}c@{,}c@{,}c@{,}c@{,}c@{,}c@{]\;\;}r}
4 & 2 & 2 & 1 & 1 & 0 & 2 & \espacio44 \\
3 & 3 & 3 & 1 & 1 & 0 & 1 & 39 \\
3 & 3 & 2 & 0 & 2 & 2 & 0 & 36
\end{array}\\ \hline\hline
%
%
13&
\FF{13}{1}&
\begin{array}{@{[}c@{,}c@{,}c@{,}c@{,}c@{,}c@{,}c@{]\;\;}r}
5 & 3 & 1 & 1 & 1 & 1 & 1 & 47 \\
5 & 2 & 2 & 2 & 0 & 1 & 1 & 44 \\
4 & 4 & 2 & 1 & 1 & 1 & 0 & 32 \\
4 & 3 & 3 & 2 & 0 & 1 & 0 & 24
\end{array}&
\FF{13}{2}&
\begin{array}{@{[}c@{,}c@{,}c@{,}c@{,}c@{,}c@{,}c@{]\;\;}r}
4 & 3 & 2 & 1 & 1 & 1 & 1 & 59 \\
4 & 2 & 2 & 2 & 1 & 2 & 0 & 56 \\
3 & 3 & 3 & 2 & 1 & 0 & 1 & 51
\end{array}&
\FF{13}{3}&
\begin{array}{@{[}c@{,}c@{,}c@{,}c@{,}c@{,}c@{,}c@{]\;\;}r}
4 & 2 & 2 & 1 & 1 & 1 & 2 & \espacio68 \\
3 & 3 & 3 & 1 & 1 & 1 & 1 & 63 \\
3 & 3 & 2 & 1 & 2 & 2 & 0 & 60
\end{array}\\ \hline\hline
%
%
&
\FF{14}{1}&
\begin{array}{@{[}c@{,}c@{,}c@{,}c@{,}c@{,}c@{,}c@{]\;\;}r}
6 & 3 & 2 & 1 & 0 & 1 & 1 & 51 \\
6 & 2 & 2 & 2 & 0 & 2 & 0 & 48 \\
5 & 4 & 3 & 1 & 0 & 1 & 0 & 15 \\
4 & 4 & 4 & 2 & 0 & 0 & 0 & 0
\end{array}&
\FF{14}{3}&
\begin{array}{@{[}c@{,}c@{,}c@{,}c@{,}c@{,}c@{,}c@{]\;\;}r}
5 & 3 & 2 & 2 & 1 & 1 & 0 & 47 \\
4 & 4 & 2 & 2 & 2 & 0 & 0 & 32 \\
4 & 3 & 3 & 3 & 1 & 0 & 0 & 27
\end{array}&
\FF{14}{5}&
\begin{array}{@{[}c@{,}c@{,}c@{,}c@{,}c@{,}c@{,}c@{]\;\;}r}
5 & 3 & 2 & 1 & 1 & 1 & 1 & \espacio71 \\
5 & 2 & 2 & 2 & 1 & 2 & 0 & 68 \\
4 & 4 & 2 & 1 & 2 & 1 & 0 & 56 \\
4 & 3 & 3 & 2 & 0 & 2 & 0 & 48
\end{array}
\\ \cline{2-7}
14&
\FF{14}{2}&
\begin{array}{@{[}c@{,}c@{,}c@{,}c@{,}c@{,}c@{,}c@{]\;\;}r}
5 & 3 & 3 & 1 & 0 & 1 & 1 & 63 \\
5 & 3 & 2 & 0 & 2 & 2 & 0 & 60 \\
4 & 4 & 3 & 0 & 2 & 0 & 1 & 48
\end{array}&
\FF{14}{4}&
\begin{array}{@{[}c@{,}c@{,}c@{,}c@{,}c@{,}c@{,}c@{]\;\;}r}
5 & 3 & 2 & 0 & 1 & 1 & 2 & 79 \\
4 & 4 & 3 & 0 & 1 & 1 & 1 & 67 \\
4 & 4 & 2 & 0 & 2 & 2 & 0 & 64
\end{array}&
\FF{14}{6}&
\begin{array}{@{[}c@{,}c@{,}c@{,}c@{,}c@{,}c@{,}c@{]\;\;}r}
4 & 3 & 2 & 1 & 2 & 1 & 1 & \espacio83 \\
4 & 2 & 2 & 2 & 2 & 2 & 0 & 80 \\
3 & 3 & 3 & 2 & 2 & 0 & 1 & 75
\end{array}
\\ \cline{2-7}
&
&
&
&
&
\FF{14}{7}&
\begin{array}{@{[}c@{,}c@{,}c@{,}c@{,}c@{,}c@{,}c@{]\;\;}r}
4 & 2 & 2 & 2 & 1 & 2 & 1 & \espacio92 \\
3 & 3 & 3 & 2 & 1 & 1 & 1 & 87 \\
3 & 3 & 2 & 2 & 2 & 2 & 0 & 84
\end{array}
\\ \hline\hline
%
%
&
\FF{15}{1}&
\begin{array}{@{[}c@{,}c@{,}c@{,}c@{,}c@{,}c@{,}c@{]\;\;}r}
6 & 4 & 1 & 2 & 1 & 1 & 0 & 44 \\
6 & 3 & 2 & 3 & 0 & 1 & 0 & 36 \\
5 & 5 & 2 & 2 & 1 & 0 & 0 & 20 \\
5 & 4 & 3 & 3 & 0 & 0 & 0 & 0
\end{array}&
\FF{15}{5}&
\begin{array}{@{[}c@{,}c@{,}c@{,}c@{,}c@{,}c@{,}c@{]\;\;}r}
5 & 4 & 2 & 2 & 0 & 1 & 1 & 68 \\
5 & 3 & 2 & 3 & 0 & 2 & 0 & 60 \\
4 & 4 & 3 & 3 & 0 & 0 & 1 & 48
\end{array}&
\FF{15}{9}&
\begin{array}{@{[}c@{,}c@{,}c@{,}c@{,}c@{,}c@{,}c@{]\;\;}r}
5 & 3 & 2 & 1 & 1 & 1 & 2 & 111 \\
4 & 4 & 3 & 1 & 1 & 1 & 1 & 99 \\
4 & 4 & 2 & 1 & 2 & 2 & 0 & 96
\end{array}
\\ \cline{2-7}
&
\FF{15}{2}&
\begin{array}{@{[}c@{,}c@{,}c@{,}c@{,}c@{,}c@{,}c@{]\;\;}r}
6 & 4 & 2 & 1 & 1 & 1 & 0 & 44 \\
6 & 3 & 3 & 2 & 0 & 1 & 0 & 36 \\
5 & 5 & 2 & 1 & 2 & 0 & 0 & 20 \\
5 & 4 & 3 & 0 & 3 & 0 & 0 & 0
\end{array}&
\FF{15}{6}&
\begin{array}{@{[}c@{,}c@{,}c@{,}c@{,}c@{,}c@{,}c@{]\;\;}r}
5 & 3 & 3 & 1 & 1 & 1 & 1 & 95 \\
5 & 3 & 2 & 1 & 2 & 2 & 0 & 92 \\
4 & 4 & 3 & 1 & 2 & 0 & 1 & 80 \\
4 & 3 & 3 & 2 & 0 & 3 & 0 & 72
\end{array}&
\FF{15}{10}&
\begin{array}{@{[}c@{,}c@{,}c@{,}c@{,}c@{,}c@{,}c@{]\;\;}r}
4 & 3 & 3 & 2 & 1 & 1 & 1 & 107 \\
4 & 3 & 2 & 2 & 2 & 2 & 0 & 104 \\
3 & 3 & 3 & 3 & 2 & 1 & 0 & 99
\end{array}
\\ \cline{2-7}
\raisebox{6.5ex}[0pt]{15}&
\FF{15}{3}&
\begin{array}{@{[}c@{,}c@{,}c@{,}c@{,}c@{,}c@{,}c@{]\;\;}r}
6 & 3 & 2 & 1 & 1 & 1 & 1 & 83 \\
6 & 2 & 2 & 2 & 1 & 2 & 0 & 80 \\
5 & 4 & 3 & 1 & 1 & 1 & 0 & 47 \\
4 & 4 & 4 & 2 & 1 & 0 & 0 & 32
\end{array}&
\FF{15}{7}&
\begin{array}{@{[}c@{,}c@{,}c@{,}c@{,}c@{,}c@{,}c@{]\;\;}r}
5 & 3 & 2 & 2 & 1 & 0 & 2 & 92 \\
4 & 4 & 3 & 2 & 1 & 0 & 1 & 80 \\
4 & 3 & 3 & 3 & 0 & 2 & 0 & 72
\end{array}&
\FF{15}{11}&
\begin{array}{@{[}c@{,}c@{,}c@{,}c@{,}c@{,}c@{,}c@{]\;\;}r}
4 & 3 & 2 & 2 & 1 & 2 & 1 & 116 \\
3 & 3 & 3 & 3 & 1 & 1 & 1 & 111 \\
3 & 3 & 2 & 3 & 2 & 2 & 0 & 108
\end{array}
\\ \cline{2-7}
&
\FF{15}{4}&
\begin{array}{@{[}c@{,}c@{,}c@{,}c@{,}c@{,}c@{,}c@{]\;\;}r}
5 & 4 & 2 & 1 & 1 & 2 & 0 & 68 \\
5 & 3 & 3 & 2 & 0 & 2 & 0 & 60 \\
4 & 4 & 3 & 0 & 3 & 1 & 0 & 48
\end{array}&
\FF{15}{8}&
\begin{array}{@{[}c@{,}c@{,}c@{,}c@{,}c@{,}c@{,}c@{]\;\;}r}
5 & 3 & 2 & 2 & 1 & 1 & 1 & 95 \\
4 & 4 & 2 & 2 & 2 & 1 & 0 & 80 \\
4 & 3 & 3 & 3 & 1 & 1 & 0 & 75
\end{array}&
&\\ \hline
\end{array}
$
\end{table}

\renewcommand{\FF}[1]{\!\!\mathcal F^{4,{#1}}_}
\renewcommand{\espacio}{\hspace{-1.1ex}}

\begin{table}[hbt]
\caption{All families $\mathcal F^{k,{n}}_j$ of almost-conjugate diagonal representations of $\Z_2^4$ of dimension $n\leq9$.}
\label{tabla:pairs4}
\vspace{3mm}
$
\begin{array}{c@{\espaciocol}c@{}l@{\espaciocol}c@{}l@{}}
n&&\hfill\textrm{Reps.}\hfill P_{4}\;\,P_{5}\;&&\hfill\textrm{Reps.}\hfill P_{4}\;\,P_{5}\;\\ \hline\hline
%
%
7&
\FF71 &
 \begin{array}{@{[}c@{,}c@{,}c@{,}c@{,}c@{,}c@{,}c@{,}c@{,}c@{,}c@{,}c@{,}c@{,}c@{,}c@{,}c@{]\;}ll}
 2& 1& 1& 1& 0& 0& 0& 1& 1& 0& 0& 0& 0&0&0&1 & 0  \\
 2& 1& 1& 1& 1& 0& 0& 0& 0& 0& 0& 0& 0&1&0&1 & 2
 \end{array}
&&
\\ \hline\hline
%
%
&
\FF81 &
 \begin{array}{@{[}c@{,}c@{,}c@{,}c@{,}c@{,}c@{,}c@{,}c@{,}c@{,}c@{,}c@{,}c@{,}c@{,}c@{,}c@{]\;}ll}
 3& 1& 1& 1& 1& 0& 0& 1& 0& 0& 0& 0& 0&0&0&3 & 0  \\
 2& 2& 2& 1& 1& 0& 0& 0& 0& 0& 0& 0& 0&0&0&0 & 0
 \end{array}
 &
\FF83 &
 \begin{array}{@{[}c@{,}c@{,}c@{,}c@{,}c@{,}c@{,}c@{,}c@{,}c@{,}c@{,}c@{,}c@{,}c@{,}c@{,}c@{]\;}ll}
2 & 1 & 1 & 1 & 1 & 0 & 0 & 1 & 0 & 1 & 0 & 0 & 0 & 0 & 0 & 3 & 2 \\
2 & 1 & 1 & 1 & 1 & 1 & 0 & 0 & 0 & 0 & 0 & 0 & 0 & 1 & 0 & 3 & 4 \\
2 & 1 & 1 & 1 & 0 & 0 & 0 & 1 & 1 & 1 & 0 & 0 & 0 & 0 & 0 & 3 & 0
 \end{array}
\\ \cline{2-5}
\raisebox{3.5ex}[0pt]{8}
&\FF82 &
 \begin{array}{@{[}c@{,}c@{,}c@{,}c@{,}c@{,}c@{,}c@{,}c@{,}c@{,}c@{,}c@{,}c@{,}c@{,}c@{,}c@{]\;}ll}
2 & 2 & 1 & 1 & 0 & 1 & 1 & 0 & 0 & 0 & 0 & 0 & 0 & 0 & 0 & 1 & 0 \\
2 & 2 & 1 & 1 & 1 & 0 & 0 & 0 & 0 & 0 & 0 & 0 & 0 & 0 & 1 & 1 & 4
 \end{array}&
&
 \\ \hline\hline
%
%
&\FF91 &
 \begin{array}{@{[}c@{,}c@{,}c@{,}c@{,}c@{,}c@{,}c@{,}c@{,}c@{,}c@{,}c@{,}c@{,}c@{,}c@{,}c@{]\;}ll}
3 & 2 & 1 & 1 & 1 & 0 & 0 & 1 & 0 & 0 & 0 & 0 &0&0&0& 3 & 0\\
2 & 2 & 2 & 1 & 2 & 0 & 0 & 0 & 0 & 0 & 0 & 0 &0&0&0& 0 & 0
 \end{array}
&
 \FF98 &
 \begin{array}{@{[}c@{,}c@{,}c@{,}c@{,}c@{,}c@{,}c@{,}c@{,}c@{,}c@{,}c@{,}c@{,}c@{,}c@{,}c@{]\;}ll}
3 & 1 & 1 & 1 & 1 & 0 & 0 & 1 & 0 & 0 & 0 & 0 & 1 & 0 & 0 & 7 & 4 \\
2 & 2 & 2 & 1 & 0 & 1 & 0 & 0 & 0 & 0 & 0 & 1 & 0 & 0 & 0 & 4 & 4 \\
2 & 2 & 2 & 1 & 0 & 0 & 1 & 0 & 1 & 0 & 0 & 0 & 0 & 0 & 0 & 4 & 0
 \end{array} \\ \cline{2-5}
&\FF92 &
 \begin{array}{@{[}c@{,}c@{,}c@{,}c@{,}c@{,}c@{,}c@{,}c@{,}c@{,}c@{,}c@{,}c@{,}c@{,}c@{,}c@{]\;}ll}
3 & 2 & 1 & 1 & 0 & 1 & 1 & 0 & 0 & 0 & 0 & 0 &0&0&0& 1 & 0 \\
3 & 2 & 1 & 1 & 1 & 0 & 0 & 0 & 0 & 0 & 0 & 0 &0&0&1& 1 & 6
 \end{array}
&\FF99 &
 \begin{array}{@{[}c@{,}c@{,}c@{,}c@{,}c@{,}c@{,}c@{,}c@{,}c@{,}c@{,}c@{,}c@{,}c@{,}c@{,}c@{]\;}ll}
3 & 1 & 1 & 1 & 1 & 1 & 0 & 0 & 0 & 0 & 0 & 0 & 0 & 1 & 0 & 3 & 6 \\
2 & 2 & 2 & 1 & 0 & 0 & 1 & 1 & 0 & 0 & 0 & 0 & 0 & 0 & 0 & 0 & 0
 \end{array} \\ \cline{2-5}
&\FF93 &
 \begin{array}{@{[}c@{,}c@{,}c@{,}c@{,}c@{,}c@{,}c@{,}c@{,}c@{,}c@{,}c@{,}c@{,}c@{,}c@{,}c@{]\;}ll}
3 & 2 & 1 & 1 & 0 & 1 & 0 & 0 & 0 & 1 & 0 & 0 &0&0&0& 3 & 0 \\
2 & 2 & 2 & 2 & 1 & 0 & 0 & 0 & 0 & 0 & 0 & 0 &0&0&0& 0 & 0
 \end{array}
&\FF9{10} &
 \begin{array}{@{[}c@{,}c@{,}c@{,}c@{,}c@{,}c@{,}c@{,}c@{,}c@{,}c@{,}c@{,}c@{,}c@{,}c@{,}c@{]\;}ll}
3 & 1 & 1 & 1 & 0 & 0 & 0 & 1 & 1 & 0 & 0 & 0 & 1 & 0 & 0 & 7 & 6 \\
2 & 2 & 2 & 1 & 0 & 0 & 1 & 0 & 0 & 0 & 0 & 0 & 0 & 0 & 1 & 4 & 8
 \end{array} \\ \cline{2-5}
9&\FF94 &
 \begin{array}{@{[}c@{,}c@{,}c@{,}c@{,}c@{,}c@{,}c@{,}c@{,}c@{,}c@{,}c@{,}c@{,}c@{,}c@{,}c@{]\;}ll}
3 & 2 & 1 & 1 & 0 & 0 & 0 & 1 & 0 & 1 & 0 & 0 &0&0&0& 2 & 0 \\
3 & 2 & 1 & 1 & 0 & 1 & 0 & 0 & 0 & 0 & 0 & 0 &0&0&1& 2 & 6
 \end{array}
&\FF9{11} &
 \begin{array}{@{[}c@{,}c@{,}c@{,}c@{,}c@{,}c@{,}c@{,}c@{,}c@{,}c@{,}c@{,}c@{,}c@{,}c@{,}c@{]\;}ll}
2 & 2 & 1 & 1 & 1 & 0 & 0 & 0 & 0 & 1 & 1 & 0 & 0 & 0 & 0 & 5 & 4 \\
2 & 2 & 1 & 1 & 1 & 0 & 0 & 1 & 0 & 0 & 0 & 0 & 1 & 0 & 0 & 5 & 6 \\
2 & 2 & 1 & 1 & 0 & 1 & 1 & 1 & 0 & 0 & 0 & 0 & 0 & 0 & 0 & 5 & 2 \\
2 & 2 & 1 & 1 & 0 & 1 & 1 & 0 & 0 & 1 & 0 & 0 & 0 & 0 & 0 & 5 & 0
 \end{array}\\ \cline{2-5}
&\FF95 &
 \begin{array}{@{[}c@{,}c@{,}c@{,}c@{,}c@{,}c@{,}c@{,}c@{,}c@{,}c@{,}c@{,}c@{,}c@{,}c@{,}c@{]\;}ll}
3 & 1 & 1 & 1 & 1 & 1 & 1 & 0 & 0 & 0 & 0 & 0 &0&0&0& 3 & 0 \\
2 & 2 & 1 & 1 & 2 & 0 & 0 & 0 & 0 & 1 & 0 & 0 &0&0&0& 0 & 0
 \end{array}
&\FF9{12} &
 \begin{array}{@{[}c@{,}c@{,}c@{,}c@{,}c@{,}c@{,}c@{,}c@{,}c@{,}c@{,}c@{,}c@{,}c@{,}c@{,}c@{]\;}ll}
2 & 2 & 1 & 1 & 1 & 0 & 0 & 0 & 0 & 0 & 1 & 0 & 1 & 0 & 0 & 8 & 4 \\
2 & 2 & 1 & 1 & 0 & 1 & 0 & 1 & 0 & 1 & 0 & 0 & 0 & 0 & 0 & 8 & 0
 \end{array} \\ \cline{2-5}
&\FF96 &
 \begin{array}{@{[}c@{,}c@{,}c@{,}c@{,}c@{,}c@{,}c@{,}c@{,}c@{,}c@{,}c@{,}c@{,}c@{,}c@{,}c@{]\;}ll}
3 & 1 & 1 & 1 & 1 & 1 & 0 & 1 & 0 & 0 & 0 & 0 &0&0&0& 7 & 0 \\
2 & 2 & 2 & 1 & 1 & 1 & 0 & 0 & 0 & 0 & 0 & 0 &0&0&0& 4 & 0
 \end{array}
&\FF9{13} &
 \begin{array}{@{[}c@{,}c@{,}c@{,}c@{,}c@{,}c@{,}c@{,}c@{,}c@{,}c@{,}c@{,}c@{,}c@{,}c@{,}c@{]\;}ll}
2 & 1 & 1 & 1 & 1 & 1 & 1 & 0 & 0 & 0 & 0 & 0 & 0 & 1 & 0 & 7 & 8 \\
2 & 1 & 1 & 1 & 1 & 0 & 1 & 1 & 0 & 1 & 0 & 0 & 0 & 0 & 0 & 7 & 6 \\
2 & 1 & 1 & 1 & 1 & 0 & 0 & 1 & 1 & 1 & 0 & 0 & 0 & 0 & 0 & 7 & 4
 \end{array}\\ \cline{2-5}
&\FF97 &
 \begin{array}{@{[}c@{,}c@{,}c@{,}c@{,}c@{,}c@{,}c@{,}c@{,}c@{,}c@{,}c@{,}c@{,}c@{,}c@{,}c@{]\;}ll}
3 & 1 & 1 & 1 & 1 & 0 & 0 & 1 & 1 & 0 & 0 & 0 & 0 & 0 & 0 & 7 & 4 \\
2 & 2 & 2 & 1 & 1 & 0 & 0 & 0 & 0 & 0 & 0 & 1 & 0 & 0 & 0 & 4 & 4
 \end{array}
&\FF9{14} &
 \begin{array}{@{[}c@{,}c@{,}c@{,}c@{,}c@{,}c@{,}c@{,}c@{,}c@{,}c@{,}c@{,}c@{,}c@{,}c@{,}c@{]\;}ll}
2 & 1 & 1 & 1 & 1 & 1 & 0 & 1 & 0 & 0 & 0 & 0 & 0 & 1 & 0 & 7 & 4 \\
2 & 1 & 1 & 1 & 0 & 0 & 0 & 1 & 1 & 1 & 0 & 0 & 0 & 1 & 0 & 7 & 0
 \end{array}
\\ \hline
\end{array}
$%
\end{table}

Now we will introduce an algorithm that allows us to find all families of $n$-di\-men\-sional almost-conjugate representations of $\Z_2^k$, for $k$ and $n$ fixed. Recall from \eqref{eq:patron} that two diagonal representations $\rho$ and $\rho'$ are almost-conjugate if and only if $c_s(\rho)=c_s(\rho')$ for all $0\leq s\leq n$. We will call the $(n+1)$-tuple $(c_0(\rho),\dots,c_n (\rho))$  the \emph{pattern} of $\rho$. The algorithm can be described as follows:

\begin{algorithm} \label{algorithm}
Let $k\geq3$ and $n\in\N$. This algorithm returns all $n$-dimensional non-equivalent diagonal representations of $\Z_2^k$ grouped into sets, where two representation are in the same set if they are almost-conjugate.
\begin{enumerate}
  \item Initialize $\mathtt{patterns}$ and $\mathtt{reps}$ as empty lists. Note that $\mathtt{reps}$ will be a list of lists of representations.
  \item Run over all $n$-dimensional diagonal representations $\rho$ of $\Z_2^k$ and obtain its pattern.
  \item If the pattern of $\rho$ is not in $\mathtt{patterns}$, add it to $\mathtt{patterns}$ at the end and add in $\mathtt{reps}$ a new entry which is a list having  $\rho$ as its only element.
      Otherwise, the pattern of $\rho$ coincides with some entry in $\mathtt{patterns}$, say the $j^{\mathrm{th}}$-entry, then look at the representations occurring in the $j^{\mathrm{th}}$-entry of $\mathtt{reps}$ and check whether any of them is equivalent to  $\rho$.
      If not, then add $\rho$ to this $j^{\mathrm{th}}$-list.
\end{enumerate}
\end{algorithm}

Tables~\ref{tabla:pairs3}, \ref{tabla:flias3}, \ref{tabla:pairs4} show some of the results obtained with the help of a computer.
They contain only representations $\rho$ such that $-\mathrm{Id}_n\notin \mathrm{Im}(\rho)$ and such that $\rho$ has no fixed vectors, (i.e.\ $q_0=0$), since we are mostly interested in manifolds (rather than orbifolds) having first Betti number zero.
For simplicity, we abbreviate by writing, for $k=3,4$ respectively,
\[
\begin{array}{c}
[q_{_1}, q_{_2}, q_{_3}, q_{_{12}}, q_{_{13}}, q_{_{23}}, q_{_{123}}]
		 =  \displaystyle{\sum_I q_{_I}\chi_{_I}}, \\ {}
[q_{_1}, q_{_2}, q_{_3}, q_{_4}, q_{_{12}}, q_{_{13}}, q_{_{14}}, q_{_{23}}, q_{_{24}},  q_{_{34}}, q_{_{123}}, q_{_{124}}, q_{_{134}}, q_{_{234}}, q_{_{1234}}]
		=  \displaystyle{\sum_I q_{_I}\chi_{_I}} .
\end{array}
\]
In the tables, for each representation we include the value of $P_{4,F}$ of primitive forms of degree $4$, when $k=3$, and the values of $P_{4,F}$ and $P_{5,F}$, when $k=4$.

Table~\ref{tabla:pairs3} shows all families of almost-conjugate representations of $\Z_2^3$ of dimension $n\leq11$. They all turn out to be pairs. On the other hand, already for $n=12$, there are $19$ families, of which 16 are pairs.
For reasons of space, in Table~\ref{tabla:flias3} we show, for $n\leq15$, all families having more than two elements, omitting the almost-conjugate pairs.
When the holonomy increases to $\Z_2^4$, the number of families increases too.
For instance, for $n=9$ there are $14$ families, shown in Table~\ref{tabla:pairs4}. For $n=10$ there are $32$ families, one of them containing $6$ representations.

\begin{rem}
(i) One can check that all pairs in Table~\ref{tabla:pairs3} can be obtained in one step by the `flip method' (Theorem~\ref{thm:metodoflip}).
However, this is not always the case when $k=3$, already for $n=12$.
Indeed, it is a simple matter to check that the first and fourth representation $[5,3,1,1,1,1,0]$ and $[4,3,3,2,0,0,0,0]$ in $\widetilde{\mathcal F}^{3,12}_1$ cannot be obtained in one flip. This is the example of minimal dimension with this property for $k=3$.
When $k=4$ the situation is very different. It is easy to check that most of the pairs in Table~\ref{tabla:pairs4}  cannot be obtained by flipping (for instance  ${\mathcal F}^{4,7}_1$, the first pair in the table).

(ii) Recall that Corollary~\ref{coro:sumprim} and Proposition~\ref{prop:prim} tell us that  $P_{4,F}$ (resp.\ $P_{4,F}+P_{5,F}$) is an
invariant of the algebra $\Lambda_F^*$ under isomorphisms when $k=3$ (resp.\ $k=4$) respectively.
We note that in all examples in the tables these numbers are different, showing different rings of invariants.
\end{rem}

\begin{rem}\label{question:open1}
\textbf{Open question.}
The above tables show many examples of families of inequivalent almost-conjugate representations such that the corresponding rings of $F$-invariants are non-isomorphic to each other.
We expect that always, given any pair of  inequivalent almost-conjugate representations, the rings of $F$-invariants in the exterior algebra are not isomorphic to each other. This happens to be true in all examples obtained computationally so far.
\end{rem}

\section{Main results}\label{sect:main}
Our next goal will be to use the results in the previous sections (Corollary~\ref{coro:sumprim}  and Theorem~\ref{thm:metodoflip}) to construct many pairs of Sunada isospectral flat manifolds of diagonal type having  very different cohomology rings.

We shall make use of the following result.

\begin{teo}\cite[Proposition 3.5]{MR2} \cite{Su} \label{thm:sunadanumbers}
Let $\Gamma$ and $\Gamma'$ be Bieberbach groups of diagonal type. Then $M_{\Gamma}$ and $M_{\Gamma'}$ are Sunada isospectral if and only if $c_{s,t}(\Gamma)=c_{s,t}(\Gamma')$ for every $s,t$ (see~\eqref{eq:Sunada}). In this case, they are $p$-isospectral for all $p$.
\end{teo}

In particular,  if $\Gamma$ and $\Gamma'$ are two Bieberbach groups of diagonal type such that the flat manifolds $M_\Gamma$ and $M_{\Gamma'}$ are Sunada isospectral, then their associated holonomy representations $\rho_\Gamma$ and $\rho_{\Gamma'}$ are almost-conjugate.
However, the equality in \eqref{eq:patron} does not suffice to imply the equality between the Sunada numbers $c_{s,t}(\Gamma)$ and $c_{s,t}(\Gamma')$ (see \eqref{eq:Sunada}).
An example of this type can be obtained by taking the two $3$-dimensional non-orientable diagonal flat manifolds with holonomy group $\Z_2^2$ called respectively \emph{first} and \emph{second amphidicosm}.
Indeed, they have the same integral representation ($\chi_0+\chi_1+\chi_2$), they are not homeomorphic, they cannot be isospectral for any flat metric on each (see \cite{RC}).
In greater dimensions there are many other examples, for instance the two different Hantzsche-Wendt manifolds in dimension $n=5$ (see~\cite{MR3}).

\begin{lema}\label{lema:Kahler}
Let $\rho=\sum_{I} q_{_I}\chi_{_I}$ be a diagonal representation of $\Z_2^k$, and let $M_\Gamma$ be a compact flat manifold with holonomy representation $\rho$. Then $M_\Gamma$ is orientable if and only if for every $1\le j \le k$,
\begin{equation} \label{eq:orientable}
\sum_{I:j\in I} q_{_I} \,\text{ is even}.
\end{equation}

Moreover,  if $q_{_I}$ is even  for every $I\subset\{1,\dots,k\}$, then $M_\Gamma$ has an invariant K\"ahler structure.
Similarly, if $q_{_I} \in 4 \Z$ for every $I\subset \{1,\dots,k\}$, then $M_\Gamma$ has an invariant hyperk\"ahler structure.
\end{lema}
\begin{proof}
The first assertion follows from the fact that  $\det(B_j) =(-1)^{\sum_{I:j\in I} q_{_I}}$ for each $j$.

Regarding the second assertion, if every $q_{_I}$ is even, then we may define a complex structure $J$ on $\R^n$ ($n=2m$)  by setting $J(e_{2i-1})= -e_{2i}, \,\, J(e_{2i}) = e_{2i-1}$ for each $1 \le i \le m$.
By \eqref{eq:rho}, this complex structure commutes with the action of the point group, hence it pushes down to a complex K\"ahler structure on $M_\Gamma$.

If furthermore each $q_{_I}$ is divisible by 4, then we can define an additional complex structure $J'$ on $\R^n$ by setting  $J'(e_{4i-3})=e_{4i-1},\,\, J'(e_{4i-2})=-e_{4i},\,\, J'(e_{4i-1})=-e_{4i-3},\,\, J'(e_{4i})=e_{4i-2}$ for $1\leq i \leq m/2$. Again this complex structure $J'$  commutes with the holonomy action and anticommutes with $J$. Therefore the pair $J, J'$ defines a hyperk\"ahler structure on $M_\Gamma$.
\end{proof}

We are now in a position to prove the main result in this paper.

\begin{teo} \label{thm:main}
For any $k\geq3$ and any $n\geq 3\, 2^{k-2}+1$ there exist explicit pairs of Bieberbach groups $\Gamma, \Gamma'$ of diagonal type of dimension $n$ with $F \cong F'\cong \Z_2^k$ such that $M_\Gamma$ and $M_{ \Gamma'}$ are Sunada (hence strongly) isospectral and their cohomology rings $H^*(M_\Gamma)$ and $H^*(M_{\Gamma'})$ are not isomorphic as graded $\Q$-algebras.

Furthermore, if $k=3,4,5$, the cohomology rings are not isomorphic as $\Q$-algebras.

\smallskip

Finally, if $n$ is even, $M_\Gamma$ is K\"ahler and $M_{\Gamma'}$ is not.
\end{teo}

\begin{proof}
We fix $k\ge3$ and $n>3\,2^{k-2}$. Set
\begin{equation}\label{eq:rhomain}
\rho=2^{k-2}\chi_{_1}+\sum_{2\in I, \, 1\notin I} 2\chi_{_I}+ q\chi_{_3},
\end{equation}
where $q=n-3\,2^{k-2}$.
One can check that  $\rho$ is faithful since it contains the characters $\chi_{_1},\chi_{_2},\chi_{_{23}},\dots,\chi_{_{2k}}$.
For this $\rho$, equation \eqref{eq:u} gives
\[
  u=\frac{1}{2^{k-2}}\big(2\, 2^{k-2}-2^{k-2}\big)=1,
\]
since there are  $2^{k-2}$ subsets $I\subset\{1,\dots,k\}$ such that $2\in I$ and $1\notin I$.
Now Theorem~\ref{thm:metodoflip} implies that $\rho$ and its flip representation
\begin{equation}\label{eq:rho'main}
  \rho'=2^{k-2}\chi_{_1}+\sum_{1\in I, \, 2\notin I} \chi_{_I}+\sum_{2\in I, \, 1\notin I} \chi_{_I}+ q\chi_{_3}
\end{equation}
are almost-conjugate.

Our next goal is to construct Bieberbach groups $\Gamma$ and $\Gamma'$ with diagonal holonomy representations $\rho$ and $\rho'$ respectively, in such a way that $M_\Gamma$ and $M_{\Gamma'}$ are Sunada isospectral manifolds.
For $I\subset \{1,\dots,k\}$, we denote by $q_{_I}$ and $q_{_I}'$ the coefficients of $\rho$ and $\rho'$ respectively and by $B_I$ and $B_I'$ the $n\times n$ diagonal matrices given by \eqref{eq:rho}.
We pick
\begin{align}
b_{1}&=\tfrac12 e_{l_1},&&\textrm{where}&
 l_1&= \textstyle{\sum\limits_{I\prec\{2\}}} q_{_I} + 1, \label{eq:b_1}\\
b_{2}&=\tfrac12 (e_{l_2}+e_{\tilde l_2}),&&\textrm{where}&
 l_2&=\textstyle{\sum\limits_{I\prec\{2,3\}}} q_{_I} +1,\quad
 \tilde l_2=\textstyle{\sum\limits_{I\prec\{3\}}}   q_{_I}+1,\label{eq:b_2}\\
b_{m}&=\tfrac12 e_{l_m},  &&\textrm{where}&
 l_m&=\textstyle{\sum\limits_{I\prec\{1\}}} q_{_I} + m-2\quad\textrm{and}\quad 3\leq m\leq k. \label{eq:b_m}
\end{align}
Now it is convenient to fix   $b_I \in \{0, \frac 12 \}^n$, as the only vector so that  $(b_I)_j \equiv \sum_{i\in I} (b_i)_j$ $\mod \Z,$, for each $I$ and for every $j$.
We define $b_I'$ for each $I$  in the same way as $b_I$,  replacing $q_{_I}$ by  $q_{_I}'$.

For each $I\subset\{1,\dots,k\}$, let $\gamma_I=B_I L_{b_I}$ and $\gamma_I'=B_I' L_{b_I'}$. Finally consider
\begin{align*}
\Gamma  =&\;\langle \gamma_{_I} :I\subset\{1,\dots,k\},\, L_{\Z^n}\rangle,\\
\Gamma' =&\;\langle \gamma_{_I}':I\subset\{1,\dots,k\},\, L_{\Z^n}\rangle.
\end{align*}
To prove that $\Gamma$ is a Bieberbach group, by \cite[Prop.\ 1.1(ii)]{MR2}, it is sufficient to check the following condition:
\begin{equation} \label{eq:conditiontorsionfree}
\text{for each $I\neq\emptyset$,  $(b_I)_j =\tfrac12$  for at least one $j$ in the space fixed by $B_I$}.
\end{equation}

In Table~\ref{tabla:Gamma} we show (in column notation) part of the matrices $B_1,\dots, B_k$ together with the vectors $b_1,\dots, b_k$.
We include only the rows $l_1, l_2, \tilde l_2, l_3,\dots,l_k$, as defined in \eqref{eq:b_1}, \eqref{eq:b_2} and \eqref{eq:b_m}, since these are the only rows having a non-zero component for at least one $b_i$,  $1\le i \le k$.
\begin{table}
\caption{The column notation for $\Gamma$.}\label{tabla:Gamma}%
\vspace{3mm}
$
\begin{array}{ccr@{}c@{}lr@{}c@{}lr@{}c@{}lr@{}c@{}lcr@{}c@{}l}
{\textrm{Character}\atop\textrm{set}}&
\rule{1mm}{0mm}{\textrm{Coor-}\atop\textrm{dinate}} \rule{1mm}{0mm}&
\multicolumn{3}{c}{\rule[-2mm]{0mm}{7mm} B_{1}} & \multicolumn{3}{c}{B_{2}} & \multicolumn{3}{c}{B_{3}}&
\multicolumn{3}{c}{\rule[-2mm]{0mm}{7mm} B_{4}} & \dots & \multicolumn{3}{c}{B_{k}} \\ \cline{3-18}
\chi_{_{1}} &l_3&-&1&        & &1&        & &1&_\frac12& &1&        &\dots& &1&\\
\chi_{_{1}} &l_4&-&1&        & &1&        & &1&        & &1&_\frac12&\dots& &1&\\
\vdots      &\vdots&&\vdots& & &\vdots&   & &\vdots&   & &\vdots&   &\ddots& &\vdots&  \\
\chi_{_{1}} &l_k&-&1&        & &1&        & &1&        & &1&        &\dots& &1&_\frac12 \\
\chi_{_{2}} &l_1& &1&_\frac12&-&1&        & &1&        & &1&        &\dots& &1& \\
\chi_{_{23}}&l_2& &1&        &-&1&_\frac12&-&1&        & &1&        &\dots& &1&\\
\chi_{_{3}} &\tilde l_2& &1&        & &1&_\frac12&-&1&        & &1&        &\dots& &1&
\end{array}
$
\end{table}

Thus, Table~\ref{tabla:Gamma} shows that the  condition \eqref{eq:conditiontorsionfree} holds for any subset $I$ having only one element.
Now assume that $I\subset\{1,\dots,k\}$ and $|I| > 1$.
If $1\notin I$, then it is clear that at least one of the coordinates $l_3,\dots,l_k$ of $b_I$ (which are in the fixed space of $B_I$)
equals $\frac12$.
Similarly, if $1\in I$ and $2\notin I$, then $(b_I)_{l_1} = \frac12$ and $B_I(e_{l_1})= e_{l_1}$; if $1\in I$, $2\in I$ and $3\in I$ then $(b_I)_{l_2} = \frac12$ and $B_I(e_{l_2})= e_{l_2}$; if $1\in I$, $2\in I$ and $3\notin I$ then $(b_I)_{\tilde l_2} = \frac12$ and $B_I(e_{\tilde l_2})= e_{\tilde l_2}$.
Thus, we have constructed Bieberbach groups $\Gamma$ and $\Gamma'$ with point groups $F$ and $F'$.

In order to check the Sunada isospectrality of $M_\Gamma$ and $M_{\Gamma'}$, we will use \eqref{eq:n_B}, \eqref{eq:n_Bhalf}, \eqref{eq:Sunada} and Theorem~\ref{thm:sunadanumbers}.
Since $\rho$ and $\rho'$ come from a flip, it is clear that \eqref{eq:n_Bflip} holds.
Furthermore,
\begin{equation*}
n_{B_{1},\frac12}=n_{B_{2},\frac12}=n_{B_{1}',\frac12}=n_{B_{2}',\frac12}=1 \qquad\text{and} \qquad
n_{B_I,\frac12}  =n_{B_I',\frac12}
\end{equation*}
for every $I\neq\{1\},\{2\}$.
Hence the Sunada numbers $c_{s,t}(\Gamma)$ and $c_{s,t}(\Gamma')$ coincide for every $0\leq t\leq s\leq n$.
We note that the first Betti number vanishes  since it follows immediately from Proposition~\ref{prop:prim} that $q_{_0}=0$ .

If $n$ is even,  it follows immediately from Lemma~\ref{lema:Kahler} that $M_\Gamma$ has an invariant K\"ahler structure since $q_{_I}$ is even for all $I$.
On the other hand, in the case of $\Gamma'$ we have that
\[
\mathcal P_{F'}^2= \left\{e_i\wedge e_j: l_3\leq i<j < l_3+2^{k-2}\right\}\bigcup \left\{ e_i\wedge e_j:\tilde l_2\leq i<j < \tilde l_2+q \right\}.
\]
Note that $\mathcal P_{F'}^2$ involves only those $e_i$ such that $i \in [\![l_3,l_3 +2^{k-2}-1]\!] \cup [\![\tilde l_2,\tilde l_2 + q-1]\!]$ and these sets do not fill all of the interval $[\![1,n]\!]$.
For instance, they do not include the index $l_2$.
This readily implies that the wedge product of $\frac n2$ times the subspace $\Lambda_{F'}^2$ cannot involve $e_{l_2}$ hence
$
\bigwedge_{1}^{n/2} \Lambda_{F'}^2=0
$
since $\Lambda^n_ {F'}$ is one-dimensional.
This implies that  $M_{\Gamma'}$ cannot admit a K\"ahler structure.

It remains only to prove the cohomology rings are not isomorphic as $\Q$-algebras for $3\le k\le 5$. For this, we will prove some inequalities between the number of primitive polynomials of degree $p$. For simplicity, throughout this proof we write $P_p$ and $P_p'$ in place of $P_{p,F}$ and $P_{p,F'}$.
Let us first prove that  $P_4'>P_4$.
From \eqref{eq:P_p} we have $P_4=\sum q_{_{I_1}}\dots q_{_{I_4}}$ where we add over all $\{I_1,\dots,I_4\}\in\mathcal A_4$ (see \eqref{eq:Ap}).
It is not difficult to see that if $\{I_1,\dots,I_4\}\in\mathcal A_4$ with $q_{_{I_1}}\dots q_{_{I_4}}>0$, then the indices $I_j$ must occur in \eqref{eq:rhomain} and cannot equal $\{1\}$ nor $\{3\}$.
Then $I_j=\{2\}\cup \widetilde I_j$ with $\widetilde I_j\subset\{3,\dots,k\}$ for every $j=1,2,3,4$, thus $q_{_{I_j}}=2$. Hence $P_4$ is $2^4$ times the number of choices of four different subsets $\widetilde I_1,\dots,\widetilde I_4 \subset \{3,\dots,k\}$ such that $\chi_{_{\widetilde I_1}}\dots \chi_{_{\widetilde I_4}}=1$.
Now, there are $2^{k-2}$ choices for $\widetilde I_1$,  $2^{k-2}-1$ choices for $\widetilde I_2$,  $2^{k-2}-2$ for $\widetilde I_3$ and $\widetilde I_1$ is determined.
This counting argument shows that
\begin{equation}\label{eq:P_4}
P_4=2^4 \frac{2^{k-2}(2^{k-2}-1)(2^{k-2}-2)}{4!}.
\end{equation}

Similarly, $P_4'=\sum q_{_{I_1}}'\dots q_{_{I_4}}'$ where we add over all elements in $\mathcal A_4$.
Now, just counting some primitive elements will show that $P_4'> P_4$.
For any $\widetilde I_1,\widetilde I_2, \widetilde I_3 \subset\{3,\dots,k\}$ with $\widetilde I_1\neq \widetilde I_2$, we take
\begin{align*}
  I_1&=\{2\}\cup\widetilde I_1,&
  I_2&=\{2\}\cup\widetilde I_2,&
  I_3&=\{1\}\cup\widetilde I_3,&
  I_4&=I_1\triangle I_2\triangle I_3,
\end{align*}
where $I\triangle J:=(I\cup J)\setminus(I\cap J)$.
One checks that $\{I_1,\dots,I_4\}\in \mathcal A_4$.
There are $\binom{2^{k-2}}{2}$ choices for the pair $\widetilde I_1$, $\widetilde I_2$.
For $\widetilde I_3$ there are $2^{k-2}$ choices, but when we consider $\widetilde I_4$ we have to divide them by two.
Now, we have to take into account the multiplicities $q_{_{I_i}}'$, which are all one except $q_{_{1}}'=2^{k-2}+1$.
Hence
\begin{equation}\label{eq:P_4'}
P_4'\geq \binom{2^{k-2}}{2} \left(q_{_1}'+\frac{2^{k-2}-2}{2!}\right)=\frac{2^{k-2}(2^{k-2}-1)}{2}2^{k-3}3.
\end{equation}
Combining \eqref{eq:P_4} and \eqref{eq:P_4'} we conclude that
\[
P_4'-P_4\geq 2^{k-2}(2^{k-2}-1)\left(\frac{2^{k-3}3}{2} - \frac{2^4(2^{k-2}-2)}{4!}\right)>0.
\]

Now we shall prove that $P_5'>P_5$ for $k>3$. Consider $\{I_1,\dots,I_5\}\in\mathcal A_5$ with $q_{_{I_1}}\dots q_{_{I_5}}>0$.
We note that $I_j\neq\{1\}$ for every $j=1,\dots,5$, since the number $1$ have to occur an even number of times.
Moreover, the index $\{3\}$ occurs once.
Then, by changing the order if necessary, we can write $I_{5}=\{3\}$ and $I_j=\{2\}\cup \widetilde I_j$ with $\widetilde I_j\subset\{3,\dots,k\}$ for $1\leq j\leq 4$. Thus $q_{_{I_{5}}}=q$ and $q_{_{I_j}}=2$ for $j=1,\dots,4$.

Hence $P_5$ equals $2^4 q$ times the number of possible choices of four different subsets $\widetilde I_1,\dots,\widetilde I_4\subset \{3,\dots,k\}$
such that $\chi_{_{\widetilde I_1}}\dots \chi_{_{\widetilde I_4}}=\chi_3$ and where no subproduct of two of them equals $\chi_0$ nor $\chi_3$.
Again, by a similar counting argument, we get
\begin{equation*}
P_5= 2^4\,q\, \frac{2^{k-2}(2^{k-2}-2)(2^{k-2}-4)}{4!}.
\end{equation*}
We can now proceed similarly as in the proof of \eqref{eq:P_4'} fixing $I_5=\{3\}$ and defining $I_4=I_1\triangle I_2\triangle I_3\triangle I_5$, obtaining
\begin{equation*}
P_5'\geq q \frac{2^{k-2}(2^{k-2}-2)}{2} \left(q_{_1}'+\frac{2^{k-2}-2}{2!}\right).
\end{equation*}
Consequently we obtain that $P_5'>P_5$ if $k>3$, as asserted.

We now prove that $P_{k+1}=0$.
Suppose there is a set $A=\{I_1,\dots,I_{k+1}\}\in\mathcal A_{k+1}$ such that $q_{_{I_{1}}}\dots q_{_{I_{k+1}}}>0$.
Note that any subset of $\{\chi_{_{I_{1}}},\dots, \chi_{_{I_{k+1}}}\}$ having $k$ elements is linearly independent in $\widehat\Z_2^k$.
By the construction of $\rho$, $\{1\}\not\in A$. When $k+1$ is even, it is clear that also $\{3\}\not\in A$, thus every $I_j\in A$ is the union of $\{2\}$ and a subset of $\{3,\dots,k\}$. This contradicts the linear independence mentioned above.
Similarly, when $k+1$ is odd, it follows that $\{3\}\in A$. The remaining $k$ elements $I_j\in A$ are such that  $2\in I_j$ and $1\not\in I_j$,
therefore they cannot be linearly independent, a contradiction.

We now show that $P_{k+1}'>0$.
When $k+1$ is even, it is clear that the product of $\chi_{_1},\chi_{_{13}}, \chi_{_{14}}, \dots ,\chi_{_{1k}}, \chi_{_{2}}, \chi_{_{23\dots k}}$ equals $\chi_{_0}$ while no subproduct of them equals $\chi_{_0}$, thus the set of the corresponding indices belong to $\mathcal A_{k+1}$.
When $k+1$ is odd, the same is true, with  $\chi_{_{3}}$ in place of $\chi_{_{13}}$.
Since all the corresponding coefficients $q_{_I}'$ in $\rho'$ are positive, the assertion follows.

The inequality $P_4<P_4'$ suffices to show that $\Lambda_F^*$ and $\Lambda_{F'}^*$ are not isomorphic as \emph{graded} algebras over~$\Q$.
To prove that the cohomology rings are not isomorphic as $\Q$-algebras, it is sufficient, by Corollary~\ref{coro:sumprim}, to show that $\sum_{p=1}^n P_p < \sum_{p=1}^n P_p'$.
We know by Proposition~\ref{prop:prim} that $P_0=P_0'=1$, $P_1=P_1'=0$ and $P_p=P_p'=0$ for $p> k+1$.
Furthermore, for $p=2,3$, $P_p=\dim(\Lambda_F^p)=\beta_p(M_\Gamma)$ and $P_p'=\dim(\Lambda_{F'}^p)=\beta_p(M_{\Gamma'})$ since $q_{_0}=0$.
Also $\beta_p(M_\Gamma)=\beta_p(M_{\Gamma'})$ for all $p$ since $\Gamma$ and $\Gamma$ are Sunada isospectral, hence $P_2=P_2'$ and $P_3=P_3'$.
Finally, since we have proven that $P_4+P_5+P_6 < P_4'+P_5'+P_6'$ when $3\leq k\leq 5$, it follows that $\Lambda_F^*$ and $\Lambda_{F'}^*$ are not isomorphic as $\Q$-algebras. This completes the proof of the theorem.
\end{proof}

\begin{rem}
\label{question:open2}
\textbf{Some open questions.}

(i) We expect that the cohomology rings of the manifolds constructed in the theorem are not isomorphic for every value of $k$, not just for $3\le k \le 5$. By similar arguments, we can still show non-isomorphism also for some values of $k>5$ but the argument becomes much more involved. We feel it would of interest to find an elegant proof valid for general $k$.

(ii) Actually, it should be possible to construct by similar methods strongly isospectral families of  arbitrarily large cardinality  having pairwise  non-isomorphic cohomology rings (provided  $k$, and hence $n$, are allowed to grow).
\end{rem}

\section{Some explicit families}\label{sect:examples}

In this last section we exhibit several strongly isospectral families in low dimensions, showing different features in their (non-isomorphic) cohomology rings.
\renewcommand{\FF}[1]{\!\!\mathcal F^{4,{#1}}_} 

\begin{ejem}\label{ex:dim7}
Here we will define $\Gamma$ and $\Gamma'$ two Bieberbach groups of dimension $n=7$ with holonomy groups $\Z_2^4$.
Their holonomy representations are those of minimal dimension obtained for $k=4$ by means of Algorithm~\ref{algorithm} (pair $\,\,\FF71$ in Table~\ref{tabla:pairs4}), namely
\begin{eqnarray*}
\rho  &=& 2\chi_{{_1}} + \chi_{{_2}} + \chi_{{_{3}}} + \chi_{{_4}} + \chi_{{_{23}}} + \chi_{{_{24}}}, \\
\rho' &=& 2\chi_{{_1}} + \chi_{{_2}} + \chi_{{_{3}}} + \chi_{{_4}} + \chi_{{_{12}}} + \chi_{{_{234}}}.
\end{eqnarray*}
The corresponding Bieberbach groups have generators $B_iL_{b_i}$ ($1\le i\le 4$) which in column notation are given by:
\\
\noindent
\rule{0.05\textwidth}{0pt}
\begin{minipage}{0.4\textwidth}
\[
\Gamma:
\begin{array}{r@{}c@{}lr@{}c@{}lr@{}c@{}lr@{}c@{}l}
\multicolumn{3}{c}{\rule[-2mm]{0mm}{7mm} B_{1}} & \multicolumn{3}{c}{B_{2}} & \multicolumn{3}{c}{B_{3}}& \multicolumn{3}{c}{B_{4}}\\ \hline
-&1&        & &1&        & &1&_\frac12& &1&        \\
-&1&        & &1&        & &1&        & &1&        \\
 &1&_\frac12  &-&1&       & &1&        & &1&        \\
 &1&        & &1&_\frac12&-&1&        & &1&        \\
 &1&_\frac12& &1&_\frac12& &1&_\frac12  &-&1&        \\
 &1&        &-&1&        &-&1&        & &1&_\frac12  \\
 &1&        &-&1&        & &1&        &-&1&
\end{array}
\]
\end{minipage}
\hfill
\begin{minipage}{0.4\textwidth}
\[
\Gamma':
\begin{array}{r@{}c@{}lr@{}c@{}lr@{}c@{}lr@{}c@{}l}
\multicolumn{3}{c}{\rule[-2mm]{0mm}{7mm} B_{1}} & \multicolumn{3}{c}{B_{2}} & \multicolumn{3}{c}{B_{3}}& \multicolumn{3}{c}{B_{4}}\\ \hline
-&1&        & &1&        & &1&_\frac12& &1&        \\
-&1&        & &1&        & &1&        & &1&        \\
 &1&_\frac12&-&1&        & &1&        & &1&        \\
 &1&        & &1&_\frac12&-&1&        & &1&        \\
 &1&_\frac12& &1&_\frac12& &1&_\frac12&-&1&        \\
-&1&        &-&1&        & &1&        & &1&_\frac12        \\
 &1&        &-&1&        &-&1&        &-&1&
\end{array}
\]\end{minipage}
\rule{0.05\textwidth}{0pt}\\

By comparison of the Sunada numbers we see that the corresponding manifolds are Sunada isospectral.
Indeed on checks that, in both cases, the non-vanishing Sunada numbers are
$ c_{5,1} = c_{3,1} = c_{1,1}= 1$, $c_{5,2}= c_{4,2}= c_{4,1}= c_{3,1}= 2$, $c_{2,1}=4$.

The rings of invariants are given in Table~\ref{tabla:exn=7}. They are not isomorphic  by Corollary~\ref{coro:sumprim}, since the total number of primitive elements equals 5 for $F$ and 7 for $F'$. Note that $\Lambda_F^2 \wedge\Lambda_F^3 =\Lambda_F^5$ while $\Lambda_{F'}^2 \wedge\Lambda_{F'}^3 = 0$.
\begin{table}
\caption{Invariants for $F$ and $F'$.}%
\label{tabla:exn=7}%
\vspace{3mm}
$
\begin{array}{c@{\espaciocol}ccccc} \hline
p &\rule[-5pt]{0pt}{14pt} {\Lambda }^p_F & P_{p,F} & P_{p,F'}& {\Lambda }^p_{F'} & \beta_p  \\ \hline\hline
0 & \operatorname{span}\{\mathbf{1}\} & 1 & 1 &
    \operatorname{span}\{\mathbf{1}\} & 1 \\ \hline
1 & 0 & 0 & 0 & 0 & 0 \\ \hline
2 & \operatorname{span}\{\,\mathbf{12}\,\} & 1 & 1 &
    \operatorname{span}\{\,\mathbf{12}\,\} & 1\\ \hline
3 & \operatorname{span}\{\,\mathbf{346},\,\mathbf{357}\,\} & 2 & 2 &
    \operatorname{span}\{\,\mathbf{136},\,\mathbf{236}\,\} & 2\\ \hline
4 & \operatorname{span}\{\,\mathbf{4567}\,\} & 1 & 1 &
    \operatorname{span}\{\,\mathbf{3457}\,\} & 1 \\ \hline
5 & \operatorname{span}\{\,12346,\,12357\,\} & 0 & 2 &
    \operatorname{span}\{\,\mathbf{14567},\,\mathbf{24567}\,\} & 2\\ \hline
6 & \operatorname{span}\{\,124567\,\} & 0 & 0 &
    \operatorname{span}\{\,123457\,\} &  1\\ \hline
7 & 0 & 0 & 0& 0 & 0\\ \hline
\end{array}
$\\[2mm]
Here we write $12$ in place of $e_1\wedge e_2$ and so on. Primitive elements are in  bold.
\end{table}
\end{ejem}

\begin{ejem}\label{ex:dim8}
We now assume that $\Gamma$ and $\Gamma'$ are as defined in the proof of Theorem~\ref{thm:main} with $k=3$ and $n=8$ ($q=2$).
The corresponding pair of diagonal representations coincides with pair $\mathcal{F}_1^ {3,8}$ in Table~\ref{tabla:pairs3}.
They are
 \begin{eqnarray*}
\rho  &=& 2\chi_{{_1}} + 2 \chi_{{_2}} + 2\chi_{{_{23}}} + 2 \chi_{{_3}} , \\
\rho' &=& 3\chi_{{_1}} + \chi_{{_{13}}} +  \chi_{{_2}} + \chi_{{_{23}}} + 2 \chi_{{_3}}.
\end{eqnarray*}
In column notation:\\
\noindent
\rule{0.15\textwidth}{0pt}
\begin{minipage}{0.3\textwidth}
\[
\Gamma:
\begin{array}{r@{}c@{}lr@{}c@{}lr@{}c@{}l}
\multicolumn{3}{c}{\rule[-2mm]{0mm}{7mm} B_{1}} & \multicolumn{3}{c}{B_{2}} & \multicolumn{3}{c}{B_{3}}\\ \hline
-&1&        & &1&        & &1&_\frac12\\
-&1&        & &1&        & &1&        \\
 &1&_\frac12&-&1&        & &1&        \\
 &1&        &-&1&        & &1&        \\
 &1&        &-&1&_\frac12&-&1&        \\
 &1&        &-&1&        &-&1&        \\
 &1&        & &1&_\frac12&-&1&        \\
 &1&        & &1&        &-&1&
\end{array}
\]
\end{minipage}
\rule{0.05\textwidth}{0pt}
\begin{minipage}{0.3\textwidth}
\[
\Gamma':
\begin{array}{r@{}c@{}lr@{}c@{}lr@{}c@{}l}
\multicolumn{3}{c}{\rule[-2mm]{0mm}{7mm} B_{1}} & \multicolumn{3}{c}{B_{2}} & \multicolumn{3}{c}{B_{3}}\\ \hline
-&1&        & &1&        & &1&_\frac12\\
-&1&        & &1&        & &1&        \\
-&1&        & &1&        & &1&        \\
-&1&        & &1&        &-&1&        \\
 &1&_\frac12&-&1&        & &1&        \\
 &1&        &-&1&_\frac12&-&1&        \\
 &1&        & &1&_\frac12&-&1&        \\
 &1&        & &1&        &-&1&
\end{array}
\]
\end{minipage}
\rule{0.15\textwidth}{0pt}\\

Then $M_\Gamma$ and $M_{\Gamma'}$ are Sunada isospectral, $H^*(M_\Gamma)\not \cong H^*(M_{\Gamma'})$ as abstract rings and furthermore $M_\Gamma$ is K\"ahler, while $M_{\Gamma'}$ is not.

We now study in some more detail the properties of these manifolds by direct computation, i.e.\ without appeal to Theorem~\ref{thm:main}.
Firstly, it is not hard to see that the $F$ (resp.$\,F'$)-invariant forms are as given in Table~\ref{tabla:exn=8}.

\begin{table}
\caption{Invariants for $F$ and $F'$.}%
\label{tabla:exn=8}%
\vspace{3mm}
$
\begin{array}{c@{\hspace{3mm}}c@{}c@{\hspace{1mm}}c@{}cc} \hline
p &\rule[-6pt]{0pt}{16pt} {\Lambda }^p_F & P_{p,F} & P_{p,F'} & {\Lambda }^p_{F'} & \beta_p  \\ \hline \hline
0 & \operatorname{span}\{\mathbf{1}\} & 1 & 1 &
    \operatorname{span}\{\mathbf{1}\} & 1 \\ \hline
1 & 0 & 0 & 0 & 0 & 0 \\ \hline
2 & \operatorname{span}\{\,\mathbf{12},\, \mathbf{34},\, \mathbf{56},\, \mathbf{78}\,\} & 4 & 4 &
    \operatorname{span}\{\,\mathbf{12},\, \mathbf{13},\, \mathbf{23},\, \mathbf{78}\,\} & 4\\ \hline
3 & \operatorname{span}
     \left\{ \begin{array}{@{}c@{}}
      \mathbf{357},\, \mathbf{367},\, \mathbf{457}\\
      \mathbf{467},\, \mathbf{358},\, \mathbf{368}\\
      \mathbf{458},\, \mathbf{468}
     \end{array} \right\} & 8 & 8 &
    \operatorname{span}
     \left\{ \begin{array}{@{}c@{}}
      \mathbf{147},\, \mathbf{247},\, \mathbf{347}\\
      \mathbf{567},\, \mathbf{148},\, \mathbf{248}\\
      \mathbf{348},\, \mathbf{568}
     \end{array} \right\} & 8\\ \hline
4 & \operatorname{span}\left\{ \begin{array}{@{}c@{}}  1234,\, 1256,\, 1278\\ 3456 ,\, 3478,\, 5678 \end{array}\right \}  & 0 & 3 &
    \operatorname{span}\left\{ \begin{array}{@{}c@{}} \mathbf{1456},\, \mathbf{2456} ,\, \mathbf{3456} \\ 1278 ,\ 1378,\, 2378\end{array}\right \}  & 6\\ \hline
5 & \operatorname{span}\left\{ \begin{array}{@{}c@{}} 12357,\, 12358\\ 12367,\, 12368\\ 12457,\, 12458\\ 12467,\, 12468\end{array} \right\} & 0 & 0 &
    \operatorname{span}\left\{ \begin{array}{@{}c@{}} 12347,\, 12348\\ 12567,\, 12568\\ 13567,\, 13568\\ 23567,\, 23568\end{array} \right\} &8 \\ \hline
6 & \operatorname{span}\left\{ \begin{array}{@{}c@{}} 123456,\, 123478\\ 125678,\, 345678\end{array} \right\} & 0 & 0 &
    \operatorname{span}\left\{ \begin{array}{@{}c@{}} 123456,\, 145678\\ 245678,\, 345678\end{array} \right\} & 4 \\ \hline
7 & 0 & 0 & 0 & 0 & 0 \\ \hline
8 & \operatorname{span}\{\,12345678\,\} & 0 & 0 &
    \operatorname{span}\{\,12345678\,\} & 1 \\ \hline
\end{array}
$\\[2mm]
Here we write $12$ in place of $e_1\wedge e_2$  and so on. Primitive elements are in  bold.
\end{table}

Using Table \ref{tabla:exn=8}, it is easy to see that
\[
\Lambda^2_{F'}\wedge \Lambda^2_{F'} = \operatorname{span} \{\,1278,\, 1378,\, 2378\,\}
\]
and ${\mathcal P}^4_{F'} = \{\, 1456,\,2456,\,3456\,\}$ thus $ P_{4,F'}=3$.
Furthermore,
\[
\Lambda^2_{F'} \wedge\Lambda^2_{F'} \wedge \Lambda^2_{F'}= 0.
\]
This clearly implies that $M_{\Gamma'}$ cannot admit a K\"ahler structure.

Now we look at $M_\Gamma$. We have
\begin{eqnarray*}
&&
\Lambda^2_{F} \wedge \Lambda^2_{F} = \Lambda^4_{F} \,,\qquad\quad
\Lambda^2_{F} \wedge \Lambda^2_{F} \wedge\Lambda^2_{F} = \Lambda^6_{F}\,,\\
&&
\Lambda^2_{F} \wedge \Lambda^2_{F}\wedge \Lambda^2_{F} \wedge \Lambda^2_{F}
= \Lambda^8_{F} = \operatorname{span}\{\,12345678\,\}.
\end{eqnarray*}
It is clear that the cohomology rings are not isomorphic as graded rings. On the other hand, since $\Lambda^2_{F}  \wedge  \Lambda^2_{F} = \Lambda^4_{F}$, this says that  ${P}_{4,F}=0$, showing  that $\Lambda_F^*$ and $\Lambda_{F'}^*$  cannot be isomorphic as algebras, by Corollary~\ref{coro:sumprim} since we have just seen that $P_{4,F'}=3$.

The complex structure on $\R^8$ given by $ J e_{2j-1} =- e_{2j}$, $J e_{2j} =e_{2j-1}$, if $j=1,2,3,4$, commutes with the holonomy action of $F$, hence it induces a K\"ahler complex structure on $M_{\Gamma}$ with K\"ahler form $\Omega = 12+34+56+78$. Note also that $\tfrac1{24} \Omega \wedge  \Omega \wedge  \Omega\wedge \Omega=  12345678.$

Since $M_\Gamma$ is K\"ahler there is a natural  action of $\mathrm{SL}(2,{\C})$ on $H^*(M_\Gamma)_{\C}\cong\Lambda^*({\C}^n)_F$ which gives the Lefschetz decomposition of $H^*(M_\Gamma)_{\C}$. We have the operators $L$, $L^*$ given by $L(X) =\Omega\wedge X$, $L^*(X)= c_p *L*(X)$ on $\Lambda^p(M_\Gamma)$, with $c_p$ a constant. A form $\eta$ is such that $L^* \eta = 0$ generates an $\mathrm{SL}(2, {\C})$-module of dimension
$n-p+1$. The decomposition into irreducible submodules is as follows
\begin{equation}
H^*(M_\Gamma)_{\C}\cong \pi_5 \oplus 3\pi_3 \oplus 8\pi_2 \oplus 2\pi_1.
\end{equation}
To verify this we note that $1$ generates the irreducible submodule of dimension 5: $\operatorname{span}\{\,1,\, \Omega,\, \Omega\wedge\Omega,\, \Omega\wedge\Omega\wedge\Omega,\, \Omega\wedge\Omega\wedge\Omega\wedge\Omega\,\}.$
Furthermore $12-34, 12-56$ and $12-78$ are $2$-forms annihilated by $L^*$ and each one generates an irreducible submodule of dimension 3.
For instance, in the case of $12-34$ we have:
\begin{eqnarray*}
\Omega \wedge(12-34)              &=&  1256 + 1278 -3456-3478\,, \\ \notag
\Omega\wedge \Omega \wedge(12-34) &=&  2\,(125678 -345678)\,.
\end{eqnarray*}
Thus $\Omega \wedge \Omega\wedge \Omega\wedge (12-34)= 0$, hence   the $\mathrm{SL}(2,{\C})$-module generated by $12-34$ has dimension 3.

Similarly, a basis for the $4$-forms annihilated by $L^*$ is $1234  + 5678 -  1256 - 3478$, $1234 + 5678 - 3456 - 1278$.
These $4$-forms generate  a trivial module.

Furthermore, we recall that the dimension of the space of the $p$-forms in $\mathrm{Ker}(L^*)$, with $p \le n$ is $\beta_0^p = \beta_p - \beta_{p-2}$, see \cite{We}, which is a topological invariant.
This implies that $\beta_0^1 =1$ $\beta_0^2 =4-1 =3$, $\beta_0^3 = 8-0 = 8$, $\beta_0^4 =6 -4 =2$ and furthermore $\beta_0^p =\beta_0^{n-p}$ for all $p$.
We see that all vectors in $\mathrm{Ker}(L^*)$ of the same degree $p\le 4$ generate an irreducible module of dimension $4-p+1$.

We  note that the doubled manifolds $M_{d\Gamma}$, $M_{d\Gamma'}$ are both isospectral K\"ahler manifolds of dimension 16, but their cohomology rings are still not isomorphic.
Furthermore, $M_{d\Gamma}$ is hyperk\"ahler but $M_{d\Gamma'}$ is not.
\end{ejem}

\begin{ejem} \label{ex:8dim24}
As a final example, we will consider a family of eight compact flat manifolds of dimension $24$ with point group isomorphic to $\Z_2^3$.
This family was also found by using the algorithm explained at the end of Section~\ref{s.construction}. The coefficients $q_{_I}^{(j)}$ of the representations $\rho_j$ for $1\leq j\leq 8$, are given in Table~\ref{tabla:rho1-8}.
We will denote by $F_j$ the point group given by Definition~\ref{def:F} of $\rho_j$.
\begin{table}
\caption{Representations $\rho_1,\dots,\rho_8$.}%
\label{tabla:rho1-8}%
\vspace{2mm}
$
\begin{array}{c|cccccccc}
&q_0&q_1&q_2&q_3&q_{12}&q_{13}&q_{23}&q_{123}\\ \hline
 \rho_1&0&10&6&3&2&1&1&1\\
 \rho_2&0&10&6&2&2&2&2&0\\
 \rho_3&0&10&5&4&3&0&1&1\\
 \rho_4&0&10&4&4&4&0&2&0\\
 \rho_5&0& 9&7&4&2&1&1&0\\
 \rho_6&0& 9&6&5&3&0&1&0\\
 \rho_7&0& 8&8&4&2&2&0&0\\
 \rho_8&0& 8&6&6&4&0&0&0
\end{array}
$
\end{table}

Now we will show that, as in Theorem~\ref{thm:main},  for each of the given diagonal representations $\rho_j$ of $\Z_2^3$ one can find $24$-dimensional vectors $b_{i}^{(j)}$ with coordinates in $\{\tfrac 12, 0\}$ ($1\le i\le 3,\; 1\le j\le 8$) such that the resulting groups $\Gamma_j=\langle B_i^{(j)} L_{b_i^{(j)}},L_{\Z^n}\rangle$ are Bieberbach groups. Indeed, one can show that these choices can be made in many different ways.

To choose the vectors $b_{i}^{(j)}$, it is convenient to fix the following order in $\widehat \Z_2^3$:
\[
\chi_0\prec \chi_1\prec \chi_2\prec\chi_3\prec\chi_{12}\prec\chi_{13}\prec\chi_{23}\prec\chi_{123}.
\]

We choose the vectors $b_I^{(j)}$ for every $1\le j \le 8$, in such a way that the first row in each character set is as in Table~\ref{table:Gamma_j} and the other rows contain no $\frac 12$.
\begin{table}
\caption{Bieberbach groups $\Gamma_1,\dots,\Gamma_8$ in column notation.}
\label{table:Gamma_j}
\vspace{0mm}
$
\begin{array}{cr@{}c@{}lr@{}c@{}lr@{}c@{}lr@{}c@{}lr@{}c@{}lr@{}c@{}lr@{}c@{}l}
{\textrm{Character}\atop\textrm{set}}&
\multicolumn{3}{c}{\rule[-2mm]{0mm}{7mm} B_{1}} & \multicolumn{3}{c}{B_{2}} & \multicolumn{3}{c}{B_{3}}&
\multicolumn{3}{c}{B_{12}}& \multicolumn{3}{c}{B_{13}} & \multicolumn{3}{c}{B_{23}}&\multicolumn{3}{c}{B_{123}} \\ \cline{2-22}
\chi_{_{1}}  &-&1&        & &1&        & &1&_\frac12&-&1&        &-&1&_\frac12& &1&_\frac12&-&1&_\frac12\\
\chi_{_{2}}  & &1&        &-&1&        & &1&_\frac12&-&1&        & &1&_\frac12&-&1&_\frac12&-&1&_\frac12\\
\chi_{_{3}}  & &1&_\frac12& &1&_\frac12&-&1&        & &1&        &-&1&_\frac12&-&1&_\frac12&-&1&        \\
\chi_{_{12}} &-&1&_\frac12&-&1&        & &1&        & &1&_\frac12&-&1&_\frac12&-&1&        & &1&_\frac12\\
\chi_{_{13}} &-&1&        & &1&        &-&1&        &-&1&        & &1&        &-&1&        & &1&        \\
\chi_{_{23}} & &1&        &-&1&        &-&1&        &-&1&        &-&1&        & &1&        & &1&        \\
\chi_{_{123}}&-&1&        &-&1&        &-&1&        & &1&        & &1&        & &1&        &-&1&
\end{array}
$
\end{table}

The coefficients $q_1^{(j)}$, $q_2^{(j)}$, $q_3^{(j)}$ and $q_{12}^{(j)}$ are positive for all $j$, which  implies that the condition in \eqref{eq:conditiontorsionfree} is verified for $\Gamma_j=\langle B_i^{(j)} L_{b_i^{(j)}}, L_{\Z^n}\rangle$.
Thus, $\Gamma_j$ is a Bieberbach group for every $1 \le j \le 8$.

We claim that these compact flat manifolds $M_{\Gamma_j}$'s are Sunada isospectral. Indeed, using Table~\ref{tabla:rho1-8} and Table~\ref{table:Gamma_j}, it is not hard to check that the numbers $n_{B_I}$'s are as in Table~\ref{table:n_BforF_j}.

\begin{table}
\caption{Numbers $n_B$ (see~\eqref{eq:n_B}) for $B$ in each point group $F_j$.}
\label{table:n_BforF_j}
\vspace{0mm}
$
\begin{array}{c|cccccccc}
&n_{B_1}&n_{B_2}&n_{B_3}&n_{B_{12}}&n_{B_{13}}&n_{B_{23}}&n_{B_{123}}\\ \hline
 \rho_1&10& 14& 18& 6& 8& 12& 4\\
 \rho_2&10& 14& 18& 4& 8& 12& 6\\
 \rho_3&10& 14& 18& 6& 8& 12& 4\\
 \rho_4&10& 14& 18& 8& 6& 12& 4\\
 \rho_5&12& 14& 18& 6& 8& 10& 4\\
 \rho_6&12& 14& 18& 8& 6& 10& 4\\
 \rho_7&12& 14& 18& 6&10&  8& 4\\
 \rho_8&12& 14& 18&10& 6&  8& 4
\end{array}\,.
$
\end{table}

This tells us that the patterns $(c_0(\rho_j),\dots,c_{24}(\rho_j))$ (see Section~2) coincide for any $1\leq j \leq 8$. The numbers $c_i(\rho_j)$ which are nonzero are $c_4=c_6=c_8=c_{10}=c_{12}=c_{14}=c_{18}=1$.
Moreover, by our choices of the vectors $b_I^{(j)}$, the non-vanishing Sunada numbers for all $\Gamma_j$ are $c_{4,1} = c_{6,1} = c_{8,1}= c_{10,2}= c_{12,2}= c_{14,1}= c_{18,2}= 1$.

To compare the cohomology rings we consider the algebra $\Lambda_{F_j}^{*}$ of $F_j$-invariants for each $j$.
Proposition~\ref{prop:prim} (i) tells us that  $P_{0,F_j}=1$, $P_{1,F_j}=0$ and $P_{p,F_j}=0$ for every $p\geq 5$ and every $j$, since $q_{0}^{(j)}=0$ and $k=3$.
Furthermore $P_{2,F_j} = \beta_2(M_{\Gamma_j})$, $P_{3,F_j} = \beta_3(M_{\Gamma_j})$.
By strong isospectrality all manifolds have the same Betti numbers, hence they have the same $P_{2}$ and $P_{3}$.
They are given by
\[
P_2=P_{2,F_8}=\tbinom{8}{2}+ \tbinom{6}{2}+ \tbinom{6}{2}+ \tbinom{4}{2}=64,\quad
P_3=P_{3,F_8}=q_{1}^{(8)}q_{2}^{(8)}q_{12}^{(8)}=192.
\]

By Corollary~\ref{coro:sumprim}, we only need to show that the values of $P_{4,F_j}$ are all different to prove that the algebras $\Lambda^*_{F_j}$ are pairwise not isomorphic.
Now, by computing $P_{4,F_j}$ by means of (ii) in  Proposition~\ref{prop:prim}, we obtain
\[
\begin{array}{c|cccccccc}
j&1&2&3&4&5&6&7&8\\ \hline
P_{4,F_j}&371& 368& 335& 320& 191& 135& 128& 0
\end{array}\,,
\]
hence our assertion follows.
\end{ejem}

\end{document}